\documentclass[a4paper]{amsart}
\usepackage[utf8]{inputenc}
\usepackage[T1]{fontenc}
\usepackage{amssymb}
\usepackage{amsmath}
\usepackage{amsthm}
\usepackage[english]{babel}
\usepackage[noadjust]{cite}
\usepackage[hang]{caption}
\usepackage{tikz}
\usepackage{tikz-cd} 
\usepackage{enumerate}
\usepackage{mathrsfs}
\usepackage{xcolor}
\usepackage{hyperref}
\usepackage{physics}
\usepackage{bbm}
\usepackage[normalem]{ulem}

\title{Spectrum of invariant measures via generic points}

\author[S. Babel]{Sejal Babel}
\address[S. Babel]{
Faculty of Mathematics and Computer Science, Jagiellonian University in Krakow, ul. \L o\-jasiewicza 6, 30-348 Krak\'ow, Poland
}\email{babelsejalm@gmail.com} 

\author[M.~E.~Can]{Melih Emin Can}
\address[M. E. Can]{
Faculty of Mathematics and Computer Science, Jagiellonian University in Krakow, ul. \L o\-jasiewicza 6, 30-348 Krak\'ow, Poland}\email{m.emincan93@gmail.com}
\author[D.~Kwietniak]{Dominik Kwietniak} 
\address[D. Kwietniak]{
Faculty of Mathematics and Computer Science, Jagiellonian University in Krakow, ul. \L o\-jasiewicza 6, 30-348 Krak\'ow, Poland}
\email{dominik.kwietniak@uj.edu.pl}
\urladdr{www.im.uj.edu.pl/DominikKwietniak/}

\author[P.~Oprocha]{Piotr Oprocha}
\address[P. Oprocha]{National Supercomputing Centre IT4Innovations, Division of the University of Ostrava,
Institute for Research and Applications of Fuzzy Modeling,
30. dubna 22, 70103 Ostrava, Czech Republic
	-- \and --
AGH University of Science and Technology, Faculty of Applied
	Mathematics, al.
	Mickiewicza 30, 30-059 Krak\'ow, Poland}
\email{piotr.oprocha@osu.cz}

\date{\today}

\newcommand{\red}{\color{red}}
\newcommand{\blue}{\color{blue}}

\newcommand{\C}{\mathbb{C}}

\newcommand{\R}{\mathbb{R}}
\newcommand{\bS}{\mathbb{S}}

\newcommand{\N}{\mathbb{N}}
\newcommand{\Z}{\mathbb{Z}}

\newcommand{\cM}{\mathcal{M}}
\newcommand{\cD}{\mathcal{D}}

\newcommand{\Ms}{\mathcal{M}_{\sigma}}
\newcommand{\MS}{\mathcal{M}_S}
\newcommand{\MT}{\mathcal{M}_T}

\newcommand{\Mse}{\mathcal{M}^e_{\sigma}}
\newcommand{\MSe}{\mathcal{M}^e_S}
\newcommand{\MTe}{\mathcal{M}^e_T}
\newcommand{\dbar}{\bar{d}}
\newcommand{\dbarm}{\bar{d}_{\mathcal{M}}}

\renewcommand{\subset}{\subseteq}
\renewcommand{\phi}{\varphi}
\newcommand{\eps}{\varepsilon}
\makeatletter
\def\blfootnote{\gdef\@thefnmark{}\@footnotetext}
\hypersetup{
    colorlinks=true,
    linkcolor=black,
    citecolor=blue,
    urlcolor=black,
    pdftitle={Overleaf Example},
    pdfpagemode=FullScreen,
    }

\theoremstyle{plain}
\newtheorem{thm}{Theorem}[section]
\newtheorem{lem}[thm]{Lemma}
\newtheorem{cor}[thm]{Corollary}

\theoremstyle{definition}
\newtheorem{defn}[thm]{Definition}
\newtheorem{rem}[thm]{Remark}

\DeclareMathOperator{\spn}{span}
\DeclareMathOperator{\Freq}{Freq}

\DeclareMathOperator{\clspn}{\overline{span}}

\DeclareMathOperator{\Gen}{Gen}

\DeclareMathOperator{\Spec}{Spec}

\DeclareMathOperator{\Av}{A}

\newcommand{\Kron}{\mathcal{K}}

\newcommand{\bshift}{X_{\mathscr{B}}}
\newcommand{\bkshift}{X_{\mathscr{B}(k)}}

\numberwithin{equation}{section}
\begin{document}
\begin{abstract} 
We describe the spectrum of an ergodic invariant measure by examining the behaviour of its generic points. 
We define Wiener--Wintner generic points for a measure 
to generalise the characterisation of generic points for 
discrete spectrum measures from Lenz et al. [Ergodic Theory and Dynamical Systems vol. \textbf{44} (2024), no. 2, 524--568]. We also study limits of sequences of generic points with respect to the Besicovitch pseudometric. This translates to results about limits of measures with respect to the metric rho-bar $\bar{\rho}$ generalising Ornstein's d-bar metric. We study how the spectrum behaves when passing to the limit, and we prove that points generic for discrete spectrum, totally ergodic, or (weakly) mixing measures, property K, and zero entropy measures form a closed set with respect to the Besicovitch pseudometric. Hence, the same holds for corresponding measures with respect to the rho-bar metric.  
Our methods have already been used to prove existence of ergodic measures with desired properties, in particular with discrete spectrum.
\end{abstract}


\maketitle 

\section{Introduction}
Let $T\colon X\to X$ be a continuous map on a compact metric space. The pair $(X,T)$ is referred to as a \emph{topological dynamical system}. 
We call $x\in X$ a \emph{generic point} for $\mu$ if for every $f\in C(X)$ we have
\begin{equation}\label{eq:ergodic-aver}
\lim_{n\to\infty}\frac1n\sum_{j=0}^{n-1}f(T^j(x))=\int_Xf\text{ d}\mu.
\end{equation}
We consider the following problem: how to determine the properties of $\mu$ given a $\mu$-generic point $x$? We would like to deduce properties of $\mu$ by looking only at the statistical behaviour of the orbit $x$. 
Note that every ergodic $T$-invariant measure always has a generic point, but non-ergodic measures need not to have them. Nevertheless,  in some topological dynamical systems every invariant measure has generic points, see \cite{ABC, CSh, DGS, GelfertKwietniak, KLO, Sigmund}. 
In \cite{Oxtoby}, Oxtoby gave a characterisation of generic points of ergodic measures. Many authors explored similar criteria for ergodicity, see \cite{GIKN, BDG, BKPLR, Weiss}. In particular, the wonderful book by Weiss \cite{Weiss} devoted a lot of space to problems of this type. 

Here, we address two variants of our motivating question. 

First, we would like to check if the spectrum of a measure could be recovered from the behaviour of a generic point. We are inspired by results obtained by Lenz and his co-authors \cite{Lenz,LSS}. We note that \cite{Lenz,LSS} uses a more general setting of ergodic topological dynamical systems given by actions of locally
compact, $\sigma$-compact Abelian groups on metrisable compact spaces. On the other hand,  \cite{Lenz,LSS} concentrate on systems with discrete (pure-point) spectrum. In particular, \cite{LSS} contains a characterisation of ergodic systems with discrete spectrum  via the existence of Besicovitch almost periodic points. Inspired by these results, and using the Wiener--Wintner ergodic theorem \cite{Assani22}, we define Wiener--Wintner $\mu$-generic points. {\red Either Wiener-Wintner or Fourier-Bohr, why both?}
Roughly speaking, a point $x$ is Fourier--Bohr generic for $\mu\in\MT(X)$ if \eqref{eq:ergodic-aver} holds and the Wiener--Wintner theorem is satisfied along the orbit of $x$ for every continuous $f\colon X\to\R$. The latter means that for every $\xi\in\bS^1$ the following limit exists
\[
\Av[f,\bar{\xi}](x)=\lim_{n\to\infty} \frac{1}{n}\sum_{j=0}^{n-1}\xi^{-j}f(T^j(x)).
\]
Regularity of a Fourier--Bohr generic point $x$ means that
\[
\norm{P_{\Kron}(f)}^2_2=\sum_{\xi\in\bS^1} \lvert \Av[f,\bar{\xi}](x)\rvert^2,
\]
where $\norm{P_{\Kron}(f)}^2_2$ is the length of the projection of $f$ onto the Kronecker factor of $\mu$. 
We prove that a Wiener--Wintner $\mu$-generic point captures the spectral properties of the measure it generates. We note that Wiener--Wintner generic points for ergodic $\mu$ with discrete spectrum are  Besicovitch almost periodic points studied in \cite{LSS}. 

Second, we would like to determine  which properties of generic points are closed with respect to \emph{Besicovitch pseudometric} $D_B$ on $(X,T)$ which is given for $x,x'\in X$ by
	\begin{equation}\label{eq:D_B}
	    D_B(x,x')=\limsup_{N\to\infty}\frac 1N\sum_{n=0}^{N-1}\rho(T^n(x),T^n(x')),
	\end{equation}
where $\rho$ is a metric on $X$.
The idea of measuring the average distance along two orbits using $D_B$ is so natural that it is hard to determine who was the first to use it. The pseudometric $D_B$ bears the name of Besicovitch because he utilised a similar pseudo-distance while studying almost periodic functions, though the terminology is not universally adopted.  
Besicovitch and related pseudometrics have proven to be useful in dynamics multiple times; see \cite{Auslander59, BFK, Downarowicz14, DownarowiczGlasner, FGJ, KKK1,KKK2,CKKK,KLO, OW04}.

Here, we connect $D_B$ with a metric $\bar{\rho}$ on the space of all $T$-invariant measures. The metric $\bar{\rho}$ generalises Ornstein's d-bar metric for finite-valued ergodic discrete-time stochastic processes. 
The aforementioned processes appear naturally in ergodic theory from
studying measure-preserving transformations through their action on finite measurable partitions of the underlying measure space. In that context, the finiteness of the partition becomes the discreteness of the state space in which the variables of the process take their values. In topological dynamics, it seems desirable to have tools applicable directly to homeomorphisms on general state spaces, rather than the pair consisting of the transformation and a fixed partition.

We study how the spectrum behaves when we take $\bar{\rho}$-limit of a sequence of measures. We show that the spectrum of the $\bar{\rho}$-limit measure is contained in the set of $\xi\in\mathbb{S}^1$ that belong to the spectra of all except possibly finitely many measures in the $\bar{\rho}$-convergent sequence. The result combines application of the theory introduced earlier with the results we provide for $\bar{\rho}$ and its connection with the Besicovitch pseudo-distance.
In addition, we prove that points generic for ergodic measures, discrete spectrum measures, (weakly) mixing measures, zero entropy measures, and measures with property K form a $D_B$-closed sets and translate these results to limits of measures with respect to  $\bar{\rho}$. For some of these properties, namely discrete spectrum  and zero-entropy systems, our results are a special case of the $D_B$-closedness of generic points for measures belonging to one of characteristic classes (see below and \cite{KKPLT23}) that was recently established by the authors of  \cite{BL}. 
While some of these aforementioned results were known in one form or another, they were primarily studied in the context of shift-invariant measures on symbolic spaces and Ornstein's d-bar metric. 
Other results seem to be new even for shift-invariant measures. Furthermore, we offer a uniform approach and extend these results beyond the setting of shift spaces. Our proofs for $\bar{\rho}$ convergence do not depend on their shift-invariant d-bar predecessors (in cases, when such results are available). 
Our methods also lead to a new proof of the rational discrete spectrum of the Mirsky measure associated with a given set of $\mathscr B$-free numbers, which highlights the utility of this new approach. We discuss some other applications in the last section.

\subsection*{Note added after a revision} 
After this paper was completed, we learnt of the work of Lenz and Strungaru~\cite{Lenz-WW}, which independently studies Wiener--Wintner type averages and their relation to the spectrum of invariant measures of topological dynamical systems. We note that Lenz and Strungaru formulate their results for general tempered F{\o}lner sequences in arbitrary $\sigma$-compact locally compact abelian groups $G$, while  throughout our paper, we work with the standard F{\o}lner sequence given by the intervals $\{0, 1, \ldots, n-1\}$ in~$G=\mathbb{Z}$.
The core observation in both papers is the same: convergence of weighted ergodic averages along the orbit of a generic point forces the existence of eigenfunctions of the Koopman representation, and one can single out a class of distinguished generic points for which the resulting eigenfunctions form an orthonormal basis of the Kronecker factor. The points we initially coined ``regular Wiener--Wintner $\mu$-generic points'' (see Definition~\ref{def:reg_WW_gen} below) coincide with the ``Wiener--Wintner points''  of $(X,G,\mu)$ in Lenz and Strungaru (Definition~5.2 in~\cite{Lenz-WW}) in the common setting $G = \mathbb{Z}$.  Both papers independently establish that the existence of such a point characterises ergodicity, that in the ergodic case almost every point is Wiener--Wintner generic, and that the discrete spectrum is characterised by points satisfying a strengthened form of this condition, see Theorem~\ref{thm:Discrete_spectrum}, which covers case $G = \mathbb{Z}$ of the characterisation of Lenz, Spindeler, and Strungaru~\cite{LSS}, and~\cite[Theorem~6.1]{Lenz-WW}, which identifies Wiener--Wintner points with generic Besicovitch almost periodic points in the pure point spectrum setting. The main differences lie in scope and focus: Lenz and Strungaru work in the greater generality of actions of locally compact abelian groups and general F{\o}lner sequences, 
and develop applications to translation bounded measure dynamical systems arising in aperiodic order. On the other hand, we develop the theory of the $\bar{\rho}$-metric on invariant measures---a generalisation of Ornstein's $\bar{d}$-metric---and establish closedness of the sets of measures with numerous dynamical properties (discrete spectrum, weak mixing, mixing, total ergodicity, property~$K$, zero entropy) under $\bar{\rho}$-convergence, with applications to $\mathscr{B}$-free shifts, specification properties, and vague specification. As we mentioned above, in the current version of this paper, we have adapted our terminology to be consistent with that of Lenz and Strungaru; in particular, what was called a ``regular Wiener--Wintner generic point'' in our original arXiv submission is now referred to as a ``Wiener--Wintner point'', in line with~\cite{Lenz-WW}. Similarly, we now write ``Fourier--Bohr $\mu$-generic point'' for what was called before ``Wiener--Wintner $\mu$-generic point''.

\section{Preliminaries}

For the remainder of this text, we assume that $(X,T)$ is an \emph{invertible} topological dynamical system; that is, $X$ is a compact metric space and $T\colon X\to X$ is a homeomorphism. All results presented here remain true if $T$ is just a continuous map from $X$ to itself. For a proof, one needs to pass to the natural extension and use the results presented here in conjunction with the well-known properties of the natural extension. The details are left to an interested reader.

Let $\rho$ be a compatible metric on $X$. Denote by $C(X)$ the separable Banach space of continuous maps $f\colon X\to \C$ endowed with the norm  $\norm{f}_\infty=\sup_{x\in X} |f(x)|$. 

Let $\cM(X)$ be the space of all Borel probability measures on $X$ endowed with the weak$^*$ topology.
We say that $\mu\in\cM(X)$ is \emph{$T$-invariant}, if $\mu(T^{-1}(B))=\mu(B)$ for every Borel set $B\subset X$. We write $\MT(X)$ for the family of all $T$-invariant measures. A measure $\mu\in\MT(X)$ is \emph{ergodic} if for every Borel set $B\subset X$ with $T^{-1}(B)=B$ either $\mu(B)=0$ or $\mu(B)=1$. The set of all ergodic measures is denoted by $\MTe(X)$. The set $\MT(X)$ is always nonempty, compact, and convex subset of $\cM(X)$ and $\MTe(X)$ is the set of extreme points of $\MT(X)$. 
For $x\in X$ we write $\delta(x)\in\mathcal{M}(X)$ for the~Dirac measure supported on~$\{x\}$.
Let $(n_k)_{k=1}^\infty$ be a sequence of positive integers with $n_k\nearrow \infty$. 
We say that $x\in X$ is \emph{quasi-generic for $\mu\in\MT(X)$ along $(n_k)_{k=1}^\infty$} if for every $f\in C(X)$ we have
\[
\lim_{k\to\infty}\frac{1}{n_k}\sum_{j=0}^{n_k-1} f(T^j(x))=\int f\text{ d}\mu. 
\]
We say that $x\in X$ is a \emph{generic point} for $\mu\in\MT(X)$ if $x$ is quasi-generic for $\mu$ along
every strictly increasing sequence $(n_k)$. We write $\Gen(\mu)$ for the set of $\mu$ generic points.

A measure $\mu$ is ergodic if and only if $\mu(\Gen(\mu))=1$ (see \cite{Oxtoby}). 
Fixing a $T$-invariant Borel probability measure $\mu$ we obtain an invertible 
measure-preserving system denoted by $(X, T,\mu)$. We omit the $\sigma$-algebra from our notation, as in our considerations the $\sigma$-algebra is always the Borel $\sigma$-algebra of $X$ or its completion. We write $L^1(X,\mu)$ (respectively, $L^2(X,\mu)$) for the corresponding Banach (respectively, Hilbert) space of equivalence classes of integrable (respectively, square-integrable) complex-valued functions on $X$. 
Let $U_T$ be the associated \emph{Koopman operator} that is, $U_T\colon L^2(X,\mu)\to L^2(X,\mu)$ is the unitary operator given by $U_T(f)=f\circ T$. 

The \emph{eigenvalues} and \emph{eigenfunctions} of $(X,T,\mu)$ are the eigenvalues and eigenfunctions of the associated Koopman operator. For the basic properties of eigenvalues and eigenvectors of $U_T$ we refer the reader to \cite[pp. 68--69]{Walters}.

For $\xi\in \bS^1$, we slightly abuse the terminology and call $\ker(U_T-\xi\text{Id})$ the eigenspace  of $\xi$. We denote by $P_\xi$ the orthogonal projection of $L^2(X,\mu)$ onto the eigenspace of $\xi$. Clearly, $\ker(U_T-\xi\text{Id})\neq\{0\}$ and $P_\xi \neq 0$ if and only if $\xi$ is an eigenvalue of $(X,T,\mu)$. 
 
We denote by $\Kron$ the closed linear subspace of $L^2(X,\mu)$ spanned by the eigenfunctions of $(X,T,\mu)$ and by 
$P_\Kron\colon L^2(X,\mu)\to\Kron$  the associated orthogonal projection. When we refer to $\Kron$, the context always makes it clear which $L^2$ space is being considered.

We say that $(X,T,\mu)$ has \emph{discrete} (or \emph{pure-point}) \emph{spectrum} if the linear span of the eigenfunctions of $U_T$ is dense in $L^2(X,\mu)$.

\section{Fourier--Bohr generic points}
Let $f\colon X\to \C$, $\xi\in\bS^1$, and $n\ge 1$. Given $x\in X$ we define
\begin{equation}\label{eq:WW}
\Av_n[f,\bar{\xi}](x)=\frac{1}{n}\sum_{j=0}^{n-1}\xi^{-j}f(T^j(x)).\end{equation}
Whenever we write $\Av[f,\bar{\xi}](x)$ we mean 
\begin{equation}
    \label{eq:Av=lim}
    \Av[f,\bar{\xi}](x)=\lim_{n\to\infty}\Av_n[f,\bar{\xi}](x),
\end{equation}
in particular, writing $\Av[f,\bar{\xi}](x)$ we always assume that the sequence $(\Av_n[f,\bar{\xi}](x))_{n=1}^\infty$ converges. 
The celebrated Wiener--Wintner ergodic theorem \cite{Assani22} states that given  a measure-preserving system 
$(X,T,\mu)$ 
and $f\in L^1(X,\mu)$ then for $\mu$-a.e. $x$ the limit $\Av[f,\bar{\xi}](x)$
exists for all $\xi\in\bS^1$, that is, there exists 
a Borel set $X_f$ with $\mu(X_f)=1$ such that for every $\xi\in\bS^1$ and $x\in X_f$ 
the sequence $\Av_n[f,\bar{\xi}](x)$ converges as $n\to\infty$.

Inspired by the concept of generic points, we define Fourier--Bohr $\mu$-generic points.

\begin{defn}
We say that $x\in X$ is a \emph{Fourier--Bohr $\mu$-generic point}  if 
$x$ is a generic point for $\mu$ 
and for every $f\in C(X)$ and $\xi\in\bS^1$ the sequence $(\Av_n[f,\bar{\xi}](x))_{n=1}^\infty$ converges as $n\to\infty$. 
\end{defn}

\begin{rem} Note that the existence of the limit in \eqref{eq:Av=lim} is not enough to associate a unique measure to a point $x$.
Every point such that the sequence $(\Av_n[f,\bar{\xi}](x))_{n=1}^\infty$ defined by \eqref{eq:WW}
converges for every $f\in C(X)$ and $\xi=1$ is a generic point for some $T$-invariant measure. Indeed, let $x\in X$ be such that $\Av[f,\bar{\xi}](x)$ exists for every $f\in C(X)$ and every $\xi\in\bS^1$. Consider the map $C(X) \ni f\mapsto \Av[f,\bar{1}](x)\in\C$. It is a continuous positive linear functional on $C(X)$, so by the Riesz representation theorem there exists a Borel probability measure $\mu\in\mathcal{M}(X)$ such that for every $f\in C(X)$ we have
\begin{equation} 
\lim_{n\to\infty}\frac1n\sum_{j=0}^{n-1} f(T^j(x))=\int f\text{ d}\mu.\label{eq:integral_WW}
\end{equation}
Observe that \eqref{eq:integral_WW} implies $\int f \text{ d}\mu =\int f\circ T \text{ d}\mu$ for every $f\in C(X)$, which means that $\mu\in\MT(X)$.
\end{rem}

The lemmas below help us decide whether a point is Fourier--Bohr $\mu$-generic. We omit their straightforward proofs. Recall that $\clspn(\cD)$ stands for the closure of the linear span of $\cD\subset C(X)$ in the $\sup$-norm $\norm{\cdot}_\infty$.

\begin{lem}\label{WWP_subset}
Let $x\in X$ and $\xi\in\bS^1$. If $(n(k))_{k=1}^\infty\subseteq\N$ is 
strictly increasing 
and $\cD\subseteq C(X)$ are such that for every $f\in\cD$ 
the sequence $(\Av_{n(k)}[f,\bar{\xi}](x))_{k=1}^\infty$ has a limit $\Phi_\xi(f)$, then the map  $\clspn(\cD)\ni f\mapsto \Phi_\xi(f)\in\C$ 
is well-defined, continuous, and linear. 
\end{lem}

\begin{lem}\label{lem:WWdense}
      Let $\cD$ be a countable set with 
      $\clspn(\cD)=C(X)$ and $\mu\in \MT(X)$. A point $x\in X$ is a Fourier--Bohr $\mu$-generic if and only if for every $f\in \cD$ and $\xi\in\bS^1$ the sequence \eqref{eq:WW} converges as $n\to\infty$ and \eqref{eq:integral_WW} holds.     
\end{lem}

Using Lemma \ref{lem:WWdense} and the Wiener--Wintner theorem we will easily see that for every ergodic $\mu\in\MT(X)$ we have that $\mu$-a.e. $x\in X$ is a Fourier--Bohr $\mu$-generic point.

\begin{lem}\label{lem:Full_set_WWRGP}
    If $(X,T)$ is a topological dynamical system and $\mu\in \MTe(X)$, then there exists 
    a Borel set $X_\mu$ with $\mu(X_\mu)=1$ such that every 
$x\in X_\mu$ is a Fourier--Bohr $\mu$-generic point.
\end{lem}

\begin{proof}
By ergodicity, 
there exists a Borel set $X_1$ consisting  of $\mu$ generic points with $\mu(X_1)=1$. 
Let $\mathcal{D}$ be a countable dense subset of $C(X)$. For every $f\in \mathcal{D}$ we use Wiener--Wintner theorem to find 
a Borel set $X_f$ with $\mu(X_f)=1$ such that for every 
 $x\in X_f$ and $\xi\in\bS^1$ the limit \eqref{eq:WW} exists. Therefore, for every \[x\in X_\mu=X_1\cap\bigcap_{f\in D}X_f\]
we have that \eqref{eq:integral_WW} holds for every $f\in\cD$ and the limit \eqref{eq:WW} exists for every $\xi\in\bS^1$ and $f\in \cD$. By Lemma~\ref{lem:WWdense} the same holds for every $f\in C(X)$. 
\end{proof}

\begin{rem}\label{rem:transl_inv}
It is easy to see that for every $k\in \Z$ and $\xi\in\bS^1$ the set of accumulation points of sequences
\[
\frac1n\sum_{j=0}^{n-1} f(T^j(x))\xi^{-j}\quad\text{and}\quad \frac1n\sum_{j=k}^{n+k-1} f(T^j(x))\xi^{-j}
\]
coincide. It follows that one of these sequences converges if and only if the other one converges. In that case, the limits coincide.
\end{rem}

The following lemma demonstrates that generic points can detect the eigenvalues of the Koopman operator.

\begin{lem}\label{lem:freq_sub_eig} 
Let $\mu\in\MT(X)$, $\xi\in \bS^1$ and $x\in X$ be a $\mu$-generic point. Then either $\Av[f,\bar{\xi}](x)$ exists and equals $0$ for every $f\in C(X)$ or $\xi$ is an eigenvalue of the Koopman operator $U_T$ on $L^2(X,\mu)$ and there exists $e_\xi\in L^2(X,\mu)$ that is a corresponding eigenfunction satisfying $\norm{e_\xi}_2\leq 1$ and
\begin{equation}\label{eq:av_proj}
\Av[f,\bar{\xi}](x)=\langle f,e_\xi\rangle=\int f\bar{e}_\xi\text{ d}\mu,
\end{equation}
provided $f\in C(X)$ is such that $\Av[f,\bar{\xi}](x)$ exists.
\end{lem}
\begin{proof} Fix $\mu\in\MT(X)$ and $\xi\in \bS^1$. Let $x\in X$ be a 
$\mu$-generic point such that for some $f_0\in C(X)$ we have that either the limit $\Av[f_0,\bar{\xi}](x)$ exists but it is not equal to $0$ or the limit $\Av[f_0,\bar{\xi}](x)$ does not exists. In any case, the sequence $(\Av_n[f_0,\bar{\xi}](x))^{\infty}_{n=1}$ has a subsequence converging to some $\gamma\in\C\setminus\{0\}$. Take $\cD$ to be a countable dense subset of $C(X)$ containing $f_0$. Let $(n_k)_{k=1}^\infty$ be the sequence of indices of that subsequence. Using a standard diagonal argument and passing to further subsequences of $(n_k)_{k=1}^\infty$ if necessary, we may assume that for every $f\in \cD$ the sequence
\[
\left(\frac{1}{n_k}\sum_{j=0}^{n_k-1} f(T^j(x))\xi^{-j}\right)_{k=1}^\infty
\]
converges as $k\to\infty$. For every $f\in \cD$ we define
\[
\Phi_\xi(f):=\lim_{k\to\infty}
\frac{1}{n_k}\sum_{j=0}^{n_k-1} f(T^j(x))\xi^{-j} \in \C.
\]
We use Lemma \ref{WWP_subset} to extend it to $\Phi_\xi\colon C(X)\to \C$.
In this way, we get a continuous linear transformation 
defined on $C(X)\subseteq L^2(X,\mu)$. 
By the Cauchy--Schwarz inequality, for every $n\ge 1$ and $f\in C(X)$ we have
\begin{multline*}
\left\lvert\frac1n\sum_{j=0}^{n-1} f(T^j(x))\xi^{-j}\right\rvert^2 \le \left(\frac1n\sum_{j=0}^{n-1} \lvert f(T^j(x))\rvert^2\right)\cdot\left( \frac1n\sum_{j=0}^{n-1} \lvert\xi^{-j}\rvert^2\right) = \frac1n\sum_{j=0}^{n-1} \lvert f(T^j(x))\rvert^2. 
\end{multline*}
Passing with $n\to \infty$ along $(n_k)_{k=1}^\infty$ and using that $x$ is $\mu$-generic, we obtain
\[
\left|\Phi_\xi(f)\right|^2=\lim_{k\to\infty}\left|\frac{1}{n_k}\sum_{j=0}^{n_k-1} f(T^j(x))\xi^{-j}\right|^2 \le \lim_{k\to\infty}\frac{1}{n_k}\sum_{j=0}^{n_k-1} |f(T^j(x))|^2=\int |f|^2\text{ d}\mu.
\]
This means that for every $f\in C(X)$ the linear map $\Phi_\xi$ satisfies 
\begin{equation}\label{eq:upperfonPhi}
 |\Phi_\xi(f)|^2\le ||f||^2_2, 
 \end{equation}
so it is a 
continuous linear functional on $C(X)$ endowed with the $L^2(X,\mu)$ norm. Now, by the Hahn-Banach Theorem (e.g. see \cite[Thm.~7.4.1.]{Narici}) 
we see that $\Phi_\xi$ extends to a  continuous linear functional on $L^2(X,\mu)$ with the same norm, in particular $\norm{\Phi_\xi}\le 1$. The density of $C(X)$ in $L^2(X,\mu)$ implies that the extension is unique and we keep denoting it by $\Phi_\xi$. By the Riesz representation theorem, there exists  a  unique $e_\xi\in L^2(X,\mu)$ such that $\norm{e_\xi}_2=\norm{\Phi_\xi}\le 1$ and for every $f\in L^2(X,\mu)$ we have
\begin{equation}
\Phi_\xi(f)=\langle f,e_\xi\rangle =\int f\bar e_\xi\text{ d}\mu.
    \label{eq:phif:product}
\end{equation} 
Clearly, $e_\xi\neq 0$, because $\Phi_\xi(f_0)=\langle f_0,e_\xi\rangle= \gamma\neq 0$. 
In order to show that $e_{\xi}$ is an eigenfunction of $U_T$  corresponding to the eigenvalue $\xi$ we need to show $e_\xi\circ T=\xi e_\xi$. Using continuity of the inner product and density of $C(X)$ in $L^2(X,\mu)$  it is enough to show that for every $f\in C(X)$ we have
\begin{equation}\label{eq:goal_1}
\langle
f,U_T(e_\xi)\rangle=
\langle
f,\xi e_\xi\rangle.    
\end{equation}
To this end, we fix $f\in C(X)$. Since $U_T$ is unitary we have
\begin{equation}
    \label{eq:prelim}
\langle
f,U_T(e_\xi)\rangle=
\langle
U^{-1}_T(f), e_\xi\rangle=\int (f\circ T^{-1})\bar e_\xi\text{ d}\mu=\Phi_\xi(f\circ T^{-1}).
\end{equation}

By the definition of $\Phi_\xi$ we have
\begin{equation}\label{eq:Phi_xi_T_inv}
    \Phi_\xi(f\circ T^{-1})=\lim_{k\to\infty}\frac{1}{n_k}\sum_{j=0}^{n_k-1} f\circ T^{-1}(T^j(x))\xi^{-j}.    
\end{equation}
Clearly, for every $n\ge 1$ it holds
\begin{equation}\label{eq:transformation}
    \frac1n\sum_{j=0}^{n-1} \left(f\circ T^{-1}(T^j(x))\right)\xi^{-j}=
    \xi^{-1} \frac1n\sum_{j=-1}^{n-2}\left( f\circ T^{j}(x)\right)\xi^{-j}.
\end{equation}
Using Remark \ref{rem:transl_inv} and noting that $\xi^{-1}=\bar\xi$ we obtain 
\begin{equation}
    \label{eq:bar-xi} \lim_{k\to\infty}\xi^{-1}\frac{1}{n_k}\sum_{j=-1}^{n_k-2}\left( f\circ T^{j}(x)\right)\xi^{-j}=\bar \xi \Phi_\xi(f).
\end{equation}

Combining 
\eqref{eq:prelim}, \eqref{eq:Phi_xi_T_inv}, \eqref{eq:transformation}, and \eqref{eq:bar-xi} we get
\[
\langle
f,U_T(e_\xi)\rangle=\Phi_\xi(f\circ T^{-1})=\bar \xi \Phi_\xi(f)=
\bar\xi \langle f,e_\xi\rangle=\langle
f,\xi e_\xi\rangle.    
\]
Hence,  \eqref{eq:goal_1} is true, so $\xi\in \Spec(X,T,\mu)$. If $A[f,\bar \xi](x)$ exists, then we have $\Phi_\xi(f)=A[f,\bar \xi](x)$, so \eqref{eq:av_proj} holds. 
\end{proof}

\begin{rem}
An alternative approach to the proof of Lemma \ref{lem:freq_sub_eig} was suggested to us by Mariusz Lemańczyk. It is based on \emph{intertwining Markov operators} associated with joinings of measure-preserving systems. Recall that a \emph{joining} of $(X,T,\mu)$ and $(Y,S,\nu)$ is a $T\times S$-invariant measure on $X\times Y$ whose projections, to the first, respectively, second coordinate is $\mu$, respectively, $\nu$. 
These operators were introduced by Vershik and since then had been used extensively, see, e.g. \cite{LPT00, LR03, ryzhikov94}. For more details, we refer to \cite{ryzhikov94} or \cite[Section 1.3]{LPT00}. 

Fix $\xi\in\bS^1$ and define a bounded arithmetic function $\mathbf{u}\colon \Z\to\C$ by $\mathbf{u}(n)=\xi^{n}$ for $n\in\Z$. Consider $\mathbf{u}$ as an element of the product space $(\bS^1)^{\Z}$. The \emph{Furstenberg system} $X_\xi$ associated with $\mathbf{u}$ is the closure of the orbit of $\mathbf{u}$ under the shift action on $(\bS^1)^{\Z}$. In other words, $X_\xi:=\overline{\{S^k(\mathbf{u}):k\in \Z\}}$, where $S\colon (\bS^1)^{\Z}\to (\bS^1)^{\Z}$ is the left-shift operator acting on $(x_k)_{k\in\Z}$ by $S((x_k)_{k\in\Z})=(x_{k+1})_{k\in\Z}$. The topological dynamical system $(X_\xi,S)$ is topologically conjugated to the closure of the orbit of the point $1\in\bS^1$ under the action of the rotation of $\bS^1$ by $\xi$. In particular, $(X_\xi,S)$ admits an unique ergodic $S$-invariant measure denoted by $\lambda_\xi$.

Let $x\in X$ be a generic point for $\mu$. Consider the point $(x,\mathbf{u})$ in the product topological dynamical system $(X\times (\bS^1)^\Z, T\times S)$. There is a strictly increasing sequence of positive integers $(n_k)_{k=1}^\infty$ such that there is a measure $\rho_\xi$ defined by
\[
\lim_{k\to \infty}\frac{1}{n_k}\sum_{j=0}^{n_k-1}\delta_{(S^j(\mathbf{u}),T^j(x))}=\rho_\xi. 
\]
It is easy to see that $\rho_\xi$ is a joining of $\lambda_\xi$ and $\mu$. Consequently, there exists an associated Markov operator $\phi_\rho\colon L^2(X_\xi,\lambda_\xi)\to L^2(X,\mu)$ 
defined for $g\in L^2(X_\xi,\lambda_\xi)$ and $f\in L^2(X,\mu)$ by 
\[\langle \phi_\rho(g),f\rangle=\int_{X_\xi\times X} g(v)\overline{f(y)}\text{ d}\rho_\xi(v,y).
\]
The Markov operator is intertwining, that is 
\begin{equation}\label{eq:intertwining_Markov}
\phi_\rho\circ U_S=U_T\circ \phi_\rho.
\end{equation}

The projection onto the zeroth coordinate, $\pi_0\colon (\bS^1)^\Z\to \bS^1$, is an eigenfunction of $(X_\xi,S,\lambda_\xi)$ (corresponding to the eigenvalue $\xi$). It follows immediately from the intertwining relationship  \eqref{eq:intertwining_Markov} that $\phi_\rho(\pi_0)$, if non-zero, is an eigenfunction of $(X,T,\mu)$ also corresponding to the eigenvalue $\xi$. 
Hence for every $f\in C(X)$ and $x$ as above we have 
\begin{multline*}
\Av[f,\bar\xi](x)=\lim_{k\to \infty}\frac{1}{n_k}\sum_{j=0}^{n_k-1}f(T^j(x))\xi^{-j}=\\=\int_{X_\xi\times X} f(y) \overline{\pi_0(v)}\text{ d}\rho_\xi(v,y) =\int_X f\overline{\phi_\rho(\pi_0)}d\mu.    
\end{multline*}
Therefore, the function $e_\xi$ in Lemma \ref{lem:freq_sub_eig} is precisely the image of the eigenfunction of $(X_\xi,S,\lambda_\xi)$ through the Markov operator associated with the joining $\rho_\xi$. 
\end{rem}

Assumptions of 
Lemma \ref{lem:freq_sub_eig} suggests the definition of frequency presented below. We are inspired by
D.~Lenz, T.~Spindeler, and N.~Strungaru \cite{LSS}, who considered similarly defined notion of frequency. For $G=\Z$ every frequency in the sense of \cite[Definition 3.4]{LSS} is also a frequency in the sense of Definition \ref{defn:Freq}.

\begin{defn}\label{defn:Freq}
We say that $\xi\in\bS^1$  
is a \emph{frequency} of $x\in X$ if there is $f\in C(X)$ such that \[
\limsup_{n\to\infty}\left|\Av_n[f,\bar\xi ](x)\right|=\limsup_{n\to\infty}\left|\frac{1}{n}\sum_{j=0}^{n-1} f(T^j(x))\xi^{-j}\right|>0.
\]
We denote the set of all frequencies of $x$ by $\Freq(x)$.
\end{defn}

\begin{cor}\label{cor:Avg_is_zero}
If $x\in X$ is a $\mu$-generic point, then 
$\Freq(x)\subseteq\Spec(X,T,\mu)$.  Furthermore, if $\xi\in\bS^1\setminus \Spec(X,T,\mu)$  and $f \in C(X)$, then $\Av[f,\bar\xi](x)$ exists and $\Av[f,\bar\xi](x)=0$.
\end{cor}
\begin{rem}\label{rem:Avg_is_zero}
Note that the topological Wiener--Wintner theorem due to Walters \cite[Corollary 2]{Walters96} (see also \cite[Theorem 3.5]{Weiss}) implies that if $\mu$ is ergodic, then for $\xi\in\bS^1$ the following conditions are equivalent:
\begin{enumerate}
    \item\label{imp-A} for every $\mu$-generic point $x\in X$  and every $f\in C(X)$ we have that $\Av[f,\bar\xi](x)$ exists and  $\Av[f,\bar\xi](x)=0$;
    \item\label{imp-B} $\xi\notin\Spec(X,T,\mu)$.
\end{enumerate}
Therefore, Corollary \ref{cor:Avg_is_zero} strengthens the implication \eqref{imp-B}$\implies$\eqref{imp-A} by removing the assumption that $\mu$ is ergodic. It follows that if $\mu$ is weakly mixing, then every $\mu$-generic point is Fourier--Bohr $\mu$-generic.
\end{rem}
The next result is an immediate consequence of Lemma \ref{lem:freq_sub_eig} and Remark \ref{rem:Avg_is_zero}. 
\begin{cor}\label{cor:eq:Xi}
Let $\mu\in\MT(X)$ and $x\in X$ be a $\mu$-generic point. Then $x$ is a Fourier--Bohr $\mu$-generic point if and only if for  every $f\in C(X)$ and $\xi\in\Spec(X,T,\mu)$ the sequence $(\Av_n[f,\bar\xi](x))_{n=1}^\infty$ converges to $\Av[f,\bar{\xi}](x)$. If any of these conditions hold, then for  every $f\in C(X)$ we have
\begin{equation}\label{eq:Xi}
\sum_{\xi\in\bS^1}\lvert \Av[f,\bar{\xi}](x)\rvert^2=\sum_{\xi\in\Spec(X,T,\mu)}\lvert \Av[f,\bar{\xi}](x)\rvert^2 =\sum_{\xi\in\Freq(x)}\lvert \Av[f,\bar{\xi}](x)\rvert^2.  
\end{equation}
\end{cor}

\begin{cor}\label{cor:Kron-Xi}
Let $\mu\in \MT(X)$, and $x\in X$ be generic for $\mu$. For every $\xi\in \Freq(x)$ let us assume that $A[f,\bar\xi](x)$ exists for every $f\in C(X)$ and take $e_\xi\in L^2(X,\mu)$ to be the function provided by Lemma \ref{lem:freq_sub_eig} so that \eqref{eq:av_proj} holds. If $f\in C(X)$, then 
\begin{equation}\label{ineq:Kron-Xi}
\lVert P_{\Kron}(f)\rVert^2_{2}\ge \sum_{\xi \in \Freq(x)} 
|\langle f, e_\xi\rangle |^2=\sum_{\xi\in\Freq(x)}\lvert \Av[f,\bar{\xi}](x)\rvert^2.    
\end{equation}
\end{cor}
\begin{proof}
It follows from Lemma  \ref{lem:freq_sub_eig} that the vectors $\{e_\xi/\norm{e_\xi}:\xi\in\Freq(x)\}$ are orthonormal and can be extended to an orthonormal basis of $\Kron$. Furthermore, for every $\xi\in \Freq(x)$ it holds
\begin{equation}\label{eq:fxisum}
|\langle f, e_\xi\rangle |^2=\norm{e_\xi}^2\left\lvert\langle f, \frac{e_\xi}{\norm{e_\xi}}\rangle \right\rvert^2\le |\langle f, \frac{e_\xi}{\norm{e_\xi}}\rangle |^2.
\end{equation}
Therefore
\[
\sum_{\xi\in\Freq(x)}\lvert \Av[f,\bar{\xi}](x)\rvert^2=\sum_{\xi \in \Freq(x)} 
|\langle f, e_\xi\rangle |^2\le \sum_{\xi \in \Freq(x)} 
|\langle f, \frac{e_\xi}{\norm{e_\xi}}\rangle |^2\le \lVert P_{\Kron}(f)\rVert^2_{2}. \qedhere \]
\end{proof}

\section{Frequencies of Fourier--Bohr Generic Points and Spectra of Measures}

In this section, we will introduce and characterise Wiener--Wintner $\mu$-generic points. Then we will connect the existence of Wiener--Wintner generic points with ergodicity and the spectrum of the associated measure. Corollary \ref{cor:Kron-Xi} suggests that we distinguish those Fourier--Bohr $\mu$-generic points for which \eqref{ineq:Kron-Xi} becomes an equality.

\begin{defn}\label{def:reg_WW_gen}Let $\mu\in\MT(X)$. 
A Fourier--Bohr $\mu$-generic point $x\in X$ is \emph{Wiener--Wintner $\mu$-generic} if for every $f\in C(X)$ it holds
\begin{equation}
    \label{eq:reg-wwg}
    \norm{P_{\Kron}(f)}^2_2=
    \sum_{\xi\in\Freq(x)}
    \lvert \Av[f,\bar{\xi}](x)\rvert^2.
\end{equation}
\end{defn}

\begin{lem}\label{lem:freq=spec}
Let $\mu\in \MT(X)$ and $x\in X$ be a $\mu$-generic point. If the set
\[E=\{e_\xi:\xi\in\Freq(x)\},\]
where $e_\xi\in L^2(X,\mu)$ is the eigenfunction of $U_T$ provided by Lemma \ref{lem:freq_sub_eig} is an orthonormal basis of $\Kron$,  then $\mu$ is ergodic and $\Freq(x)=\Spec(X,T,\mu)$.
\end{lem}
\begin{proof}
By Remark \ref{rem:Avg_is_zero}, we have $\Freq(x)\subseteq\Spec(X,T,\mu)$. Now, if there were $\hat{\xi}\in \Spec(X,T,\mu)\setminus\Freq(x)$, then the corresponding eigenfunction $\hat{e}$ would be orthogonal to the space spanned by $\{e_\xi \colon \xi\in \Freq(x)\}=\Kron$, which is impossible since $\hat{e}\in\Kron$. Furthermore, for each $\xi\in\Spec(X,T,\mu)=\Freq(x)$ the eigenspace $\ker(U_T-\xi\text{Id})$ is one-dimensional. In particular, the eigenspace corresponding to the eigenvalue $1$ is one-dimensional, which means that all $T$-invariant functions are $\mu$-a.e. constant. This implies ergodicity. 
\end{proof}

Lemma \ref{lem:equiv_regular_points} provides us with some useful characterisations of a Wiener--Wintner $\mu$-generic point. 
\begin{lem}\label{lem:equiv_regular_points}

Let $\mu\in \MT(X)$ and $x\in X$ be a generic point for $\mu$ such that for every $\xi\in \Freq(x)$ and $f\in C(X)$ the sequence $(\Av_n[f,\bar\xi](x))_{n=1}^\infty$ converges (i.e. $A[f,\bar\xi](x)$ exists). Let $e_\xi\in L^2(X,\mu)$ be the function provided by Lemma \ref{lem:freq_sub_eig}. 
Then the following conditions are equivalent: 
\begin{enumerate}
\item \label{c:i}The point $x$ is a Wiener--Wintner $\mu$-generic point.
\item \label{c:iii} For every $f\in C(X)$ we have $\displaystyle\lVert P_{\Kron}(f)\rVert_{2}^2=\sum_{\xi\in \Freq(x)}\lvert \langle f,e_\xi\rangle \rvert^2$. 
\item \label{c:iv} The set $\{e_\xi \colon \xi\in \Freq(x)\}$ forms an orthonormal basis for $\Kron$.
\end{enumerate}
\end{lem}
\begin{proof} 
\eqref{c:i} $\implies$ \eqref{c:iii}
It follows from substituting \eqref{eq:av_proj} into \eqref{eq:reg-wwg}.%
\eqref{c:iii} $\implies$ \eqref{c:iv} 
Let $f\in C(X)$. By our assumption and $\norm{e_\xi}\le 1$ 
we have 
\begin{multline*}
    \lVert P_{\Kron}(f)\rVert_{2}^2=\sum_{\xi\in \Freq(x)}|\langle f, e_\xi\rangle |^2=\sum_{\xi\in \Freq(x)}\norm{e_\xi}^2\left\lvert\langle f, \frac{e_\xi}{\norm{e_\xi}}\rangle \right\rvert^2\le\\
\sum_{\xi\in \Freq(x)}\left\lvert\langle f, \frac{e_\xi}{\norm{e_\xi}}\rangle \right\rvert^2
\le \lVert P_{\Kron}(f)\rVert_{2}^2.
\end{multline*}
Hence, $\norm{e_\xi}^2=1$ for every $\xi\in\Freq(x)$. By Parseval's 
identity $P_\Kron(f)\in\spn\{e_\xi:\xi\in \Freq(x)\}$ for every $f\in C(X)$, so $\clspn\{e_\xi:\xi\in \Freq(x)\}=\Kron$ as needed.

\eqref{c:iv} $\implies$ \eqref{c:i}  By Lemma \ref{lem:freq=spec}, we have $\Freq(x)=\Spec(X,T,\mu)$. 
Therefore using our assumptions, we see that $\Av[f,\bar{\xi}](x)$ exists for every $\xi\in \Spec(X,T,\mu)$ and $f\in C(X)$. By Corollary \ref{cor:eq:Xi} the point $x$ is Fourier--Bohr $\mu$-generic. By Parseval's 
identity, for every $f\in C(X)$ it holds
\begin{equation}\label{eq:basis}
\norm{P_{\Kron}(f)}^2_2=\sum_{\xi\in \Freq(x)} \lvert \langle f, e_\xi \rangle \rvert^2.   
\end{equation}
Substituting \eqref{eq:av_proj} from Lemma \ref{lem:freq_sub_eig} into \eqref{eq:basis}, we get that $x$ is regular. 
\end{proof}
Note that it is enough to check the condition of regularity \eqref{eq:reg-wwg} only for countably many functions.
\begin{lem}\label{cor:dense-WWRG}
Let $\cD\subset C(X)$ be a countable set with  $\clspn(\cD)=C(X)$ 
and $\mu\in \MT(X)$. A point $x\in X$ is a Wiener--Wintner $\mu$-generic if and only if $x$ is a Fourier--Bohr $\mu$-generic and \eqref{eq:reg-wwg} is satisfied for every $f\in \cD.$
\end{lem}

\begin{thm}\label{thm:WW-reg-erg}
Let  $\mu\in\MT(X)$. 
The following conditions are equivalent:
\begin{enumerate}
    \item \label{reg-erg-i} The measure $\mu$ is ergodic. 
    \item \label{reg-erg-ii} There exists a Borel set $\hat{X}_\mu$ with $
    \mu(\hat{X}_\mu)=1$ such that every $x\in \hat{X}_\mu$ is a Wiener--Wintner $\mu$-generic point.
 \item \label{reg-erg-iii} There exists a Wiener--Wintner $\mu$-generic point.  \end{enumerate}    
\end{thm}
\begin{proof}
\eqref{reg-erg-i}$\implies$\eqref{reg-erg-ii}
For every $\xi\in\Spec(X,T,\mu)$, we pick a corresponding normalised eigenfunction $e_\xi\in L^2(X,\mu)$. By ergodicity, $e_\xi$ spans the eigenspace $\ker(U_T-\xi\text{Id})$. Hence, 
$\{e_\xi:\xi\in\Spec(X,T,\mu)\}$ is an orthonormal basis of $\Kron=\Kron(\mu)$. 
Thus, for every $f\in L^2(X,\mu)$ we have
\begin{equation}\label{eq:Kron}
 \norm{P_{\mathcal{\Kron}}(f)}^2_2=\sum_{\xi\in\Spec(X,T,\mu)}|\langle f,e_\xi\rangle|^2.
   \end{equation}
   
Let $X_\mu$ be the full measure set of Fourier--Bohr $\mu$-generic points obtained from  Lemma \ref{lem:Full_set_WWRGP}.
For every $\xi\in \Spec(X,T,\mu)$, let $X_\xi$ be a full measure set such that $|e_\xi(x)|=1$ for every $x\in X_\xi$. 
Pick a countable and dense set $\cD\subseteq C(X)$. For every $f\in\cD$ and $\xi\in\Spec(X,T,\mu)$ we use the ergodic theorem to find a full measure set $X_{\xi,f}\subseteq X_\xi$ such that for every $x\in X_{\xi,f}$ and for every $j\ge 1$ we have  $\overline{e_\xi(T^j(x))}=\xi^{-j}\bar{e}_\xi(x)$ and 
\[
\lim_{n\to\infty}\frac1n\sum_{j=0}^{n-1} f(T^j(x))\overline{e_\xi(T^j(x))} = \int f\bar{e}_\xi\text{ d}\mu= \langle f, e_\xi \rangle . 
\]
Furthermore, we note that for every $x\in X_{\xi,f}$ it holds
\[
\frac1n\sum_{j=0}^{n-1} f(T^j(x))\overline{e_\xi(T^j(x))}=\bar{e}_\xi(x)\frac1n\sum_{j=0}^{n-1} \xi^{-j}f(T^j(x)).
\]
Hence for every $x\in X_{\xi,f}$ we get
\begin{equation}\label{eq:Av-inner}
\Av[f,\bar\xi](x)=\lim_{n\to\infty}\frac1n\sum_{j=0}^{n-1} \xi^{-j}f(T^j(x))=\langle f, e_\xi\rangle e_\xi(x). 
\end{equation}

Since $|e_\xi(x)|=1$ for $x\in X_{\xi,f}$, by taking the absolute values and squaring the sides of the equation \eqref{eq:Av-inner}   we obtain
\begin{equation}
    \label{eq:squared}
|\Av[f,\bar\xi](x)|^2
=|\langle f,e_\xi\rangle|^2.
\end{equation}
Define 
\[
\hat{X}_{\mu}=X_{\mu}\cap \bigcap_{\xi\in\Spec(X,T,\mu)}\bigcap_{f\in\cD} X_{\xi,f}.
\]
Clearly, $\mu(\hat{X}_{\mu})=1$ so it remains to prove that each point in $\hat{X}_{\mu}\subseteq X_\mu$ is regular.
By \eqref{eq:Kron}, \eqref{eq:squared}, and Remark \ref{rem:Avg_is_zero} for every $x\in \hat{X}_{\mu}$ and every $f\in\cD$ we have
\[
 \norm{P_{\Kron}(f)}^2_2=\sum_{\xi\in\Spec(X,T,\mu)}|\langle f,e_\xi\rangle|^2=
 \sum_{\xi\in\Freq(x)}|\Av[f,\bar\xi](x)|^2.
\]
By Lemma \ref{cor:dense-WWRG} every point in $\hat{X}_{\mu}$ is Wiener--Wintner $\mu$-generic.

\eqref{reg-erg-ii}$\implies$\eqref{reg-erg-iii} This is immediate.

\eqref{reg-erg-iii}$\implies$\eqref{reg-erg-i} 
This follows from Lemma \ref{lem:equiv_regular_points} and Lemma \ref{lem:freq=spec}.  
\end{proof}

Theorem \ref{thm:WW-reg-erg} leads to our next result, which enables us to constrain the spectrum using some knowledge on the behaviour of generic points. 

\begin{cor}\label{cor:criterion-Xi}
Let $\mu\in\MT(X)$ and let $\Xi$ be a subset of $\bS^1$.
The following conditions are equivalent:
\begin{enumerate}
    \item \label{Xi-i} The measure $\mu$ is ergodic and $\Spec(X,T,\mu)\subseteq\Xi$. 
    \item \label{Xi-ii} There exists a Borel set $X'$ with $\mu(X')=1$  such that 
    every $x\in X'$ is a Wiener--Wintner $\mu$-generic point  and for every $f\in C(X)$ we have
\begin{equation}\label{eq:reg-wwg-xi}
 \norm{P_{\Kron}(f)}^2_2=\sum_{\xi\in\Xi}|\Av[f,\bar{\xi}](x)|^2.    
\end{equation}
 \item \label{Xi-iii} There exists a $\mu$-generic point $x\in X$ and a countable dense set $\cD\subseteq C(X)$ such that for every $f\in \cD$ and $\xi\in \Xi$ the sequence $(\Av_n[f,\bar{\xi}](x))_{n=1}^\infty$ converges and equality \eqref{eq:reg-wwg-xi} holds.
 \item \label{Xi-iv} 
There is a Wiener--Wintner  $\mu$-generic point $x$  with  $\Freq(x)\subseteq\Xi$.
\end{enumerate}    
\end{cor}
\begin{proof} \eqref{Xi-i}$\implies$\eqref{Xi-ii} Use Theorem \ref{thm:WW-reg-erg} to find a full measure set $\hat{X}_{\mu}$ of Wiener--Wintner $\mu$-generic points. For every $x\in \hat{X}_{\mu}$, we use Lemma  \ref{lem:freq=spec} to substitute $\Freq(x)$ with $\Spec(X,T,\mu)$ in Lemma \ref{lem:equiv_regular_points} Since $\Spec(X,T,\mu)\subseteq \Xi$ and $\Av[f,\bar{\xi}](x)=0$ for $\xi\notin\Freq(x)$, we obtain 
\begin{equation*}
\begin{aligned}
\norm{P_{\Kron}(f)}^2_2 = \sum_{\xi\in\Spec(X,T,\mu)}|\Av[f,\bar{\xi}](x)|^2 
&\le \sum_{\xi\in\Xi}|\Av[f,\bar{\xi}](x)|^2 \\
&= \sum_{\xi\in\Xi\cap\Freq(x)}|\Av[f,\bar{\xi}](x)|^2 
\le \norm{P_\Kron(f)}_2^2.
\end{aligned}
\end{equation*} 
Hence \eqref{eq:reg-wwg-xi} holds for every $x\in\hat{X}_{\mu}$.

\eqref{Xi-ii}$\implies$\eqref{Xi-iii} This is immediate.

\eqref{Xi-iii}$\implies$\eqref{Xi-iv} Since $\cD$ is dense in $C(X)$, using Lemma \ref{WWP_subset} we see that for every $\xi\in \Xi$ and $f\in C(X)$ the limit $A[f,\bar{\xi}](x)$ exists. 
By proof of Corollary \ref{cor:Kron-Xi} and our assumption we have that
\begin{equation}
\begin{aligned}
\lVert P_{\Kron}(f)\rVert^2_{2}\ge \sum_{\xi\in\Freq(x)}\lvert \langle f, e_\xi\rangle\rvert^2  &\ge \sum_{\xi\in\Freq(x)\cap\Xi}\lvert \langle f,e_\xi \rangle \rvert^2 \\
&=\sum_{\xi\in\Freq(x)\cap\Xi}\lvert \Av[f,\bar{\xi}](x)\rvert^2=\lVert P_{\Kron}(f)\rVert^2_{2}.
\end{aligned}
\end{equation}

It follows that \[
\sum_{\xi\in\Xi}\lvert\langle f,e_\xi\rangle \rvert^2=\lVert P_{\Kron}(f)\rVert^2_{2}=\sum_{\xi\in\Freq(x)}\lvert \langle f,e_\xi\rangle \rvert^2,\]
 so $\Freq(x)\subset \Xi$, and therefore, by Lemma \ref{lem:equiv_regular_points}, we see that $x$ must be a Wiener--Wintner $\mu$-generic point.

\eqref{Xi-iv}$\implies$\eqref{Xi-i} Since $x$ is a Wiener--Wintner $\mu$-generic point, we have that $\mu$ is ergodic by Theorem \ref{thm:WW-reg-erg} and $\Spec(X,T,\mu)=\Freq(x)\subseteq\Xi$ by Lemma \ref{lem:equiv_regular_points} and Lemma \ref{lem:freq=spec}.  
\end{proof}

For uniquely ergodic topological dynamical systems with continuous eigenfunctions, all points turn out to be Wiener--Wintner generic for the unique invariant measure of the system.

\begin{thm}\label{lem:uniformAV}
If $(X,T)$ is a uniquely ergodic topological dynamical system with continuous eigenfunctions, then every $x\in X$  is a Wiener--Wintner $\mu$-generic point, where $\mu$ stands for the unique $T$-invariant measure. 
\end{thm}   
\begin{proof}
Fix $f\in C(X)$. Recall that $P_\xi(f)$ denotes the projection on the eigenspace of $\xi$ when $\xi\in \Spec(X,T,\mu)$ and $P_\xi(f)=0$ otherwise. 

By the uniform version of the Wiener--Wintner theorem \cite[Theorem 1.1]{Robinson94}, we obtain that for every $x\in X$, 
for every $f\in C(X)$, and $\xi\in\bS^1$ we have
\begin{equation*}
    \Av[f,\bar\xi](x)=\lim_{n\to \infty}\Av_n[f,\bar{\xi}](x)=P_\xi(f)(x). 
\end{equation*}
Furthermore, $P_\xi(f)$ is a continuous function. Hence, every $x\in X$ is Fourier--Bohr $\mu$-generic, so $\Freq(x)=\Spec(X,T,\mu)$ by Lemma \ref{lem:freq=spec}. 
For every $k\ge 0$ and $x\in X$ we have
\[
\lvert P_\xi(f)(T^k(x)) \rvert^2=\lvert \xi^{-k} P_\xi(f)(x) \rvert^2=\lvert P_\xi(f)(x) \rvert^2,
\]
so $x\mapsto \lvert P_\xi(f)(x) \rvert^2$ is a continuous $T$-invariant function. By unique ergodicity, we see that for every $x\in X$ it holds
\[\lvert P_\xi(f)(x) \rvert^2=\int \lvert P_\xi(f)(y) \rvert^2 \text{ d}\mu(y)=\norm{P_\xi(f)}_2^2.
\]
Therefore
\[
\sum_{\xi\in \bS^1}|\Av[f,\bar\xi](x)|^2=\sum_{\xi\in \bS^1}\norm{P_\xi(f)}^2_2=\sum_{\xi\in\Spec(X,T,\mu)}\norm{P_\xi(f)}_2^2=\norm{P_\Kron(f)}_2^2.\qedhere
\]\end{proof}
\begin{rem}
Conversely, assume that $(X,T)$ is minimal and that every $x\in X$ is a Fourier--Bohr generic point. Then $(X,T)$ is uniquely ergodic and all its eigenfunctions are continuous. Indeed, suppose that there exist $\xi_0\in \bS^1$ and $h\in L^2(X,\mu)$ such that $h\circ T=\xi_0 h$, and $h$ is not equal $\mu$ a.e to any continuous function. By~\cite[Corollary 9]{Walters96} there exist $f_0\in C(X)$ and $x_0\in X$ such that the limit
\[
\lim_{n\to \infty}\frac{1}{n}\sum_{j=0}^{n-1}f_0(T^j x_0)\xi_0^{-j}
\]
fails to exist, contradicting the assumption that every $x\in X$ is a Fourier--Bohr $\mu$-generic point.
\end{rem}

\section{A characterisation of discrete spectrum and Besicovitch almost periodic points}

As mentioned in the introduction, we are inspired by \cite{LSS}.
In this section, we recover a special case $G=\Z$ of a result of Lenz, Splinder and Strungaru \cite[Theorem 4.7]{LSS} characterising systems with discrete spectrum. To see that Theorem \ref{thm:Discrete_spectrum} below contains \cite[Theorem 4.7]{LSS} when $G=\Z$, one must note that by \cite[Proposition 4.2]{LSS},  the Besicovitch almost periodic points are the same as  Wiener--Wintner $\mu$-generic points when we replace $\norm{P_\Kron(f)}_2$ with $\norm{f}_2$ 
 in \eqref{eq:reg-wwg}. Theorem \ref{thm:Discrete_spectrum} follows directly from Corollary \ref{cor:criterion-Xi}.

  \begin{thm}\label{thm:Discrete_spectrum}
    Let 
    $\mu\in \MT(X)$ and $\Xi\subseteq \bS^1$. The following conditions are equivalent:
    
       \begin{enumerate}
       \item\label{thm:DiscEquiv:1} The measure $\mu$ is ergodic, has discrete spectrum and $\Spec(X,T,\mu)\subset \Xi.$

      \item\label{thm:DiscEquiv:2} There exists a Borel set $X'$ with $\mu(X')=1$ such that every $x\in X'$ is a Wiener--Wintner $\mu$-generic point and for every $f\in C(X)$ we have  \begin{equation}\label{eq:discrete_spectrum}
        \int \lvert f\rvert^2\text{ d}\mu = \sum_{\xi\in\Xi} |\Av[f,\bar\xi](x)|^2.\end{equation} 
\item \label{thm:DiscEquiv:3+} There exists a  $\mu$-generic point $x\in X$ and a countable dense set $\cD\subset C(X)$ such that for every $f\in \cD$ and $\xi\in \Xi$ the limit $\Av[f,\bar\xi](x)$ exists and the equality \eqref{eq:discrete_spectrum} holds.
    \item \label{thm:DiscEquiv:3} There exists a Wiener--Wintner $\mu$-generic point $x\in X$ such that for every $f\in C(X)$ the equality \eqref{eq:discrete_spectrum} holds and $\Freq(x)\subseteq\Xi$.
      \end{enumerate}   

\end{thm}

\section{Besicovitch pseudometric and its relatives}\label{sec:Besicovitch}
The aim of this section is to state a few technical results about Besicovitch pseudometric and related notions that we need later. Although several of these findings have been previously published, their original form does not meet our requirements. Hence, we reformulate them here.

Let $(n_k)_{k=1}^\infty$ be a strictly increasing sequence of positive integers. 
The \emph{upper/lower asymptotic density} $d^*_{(n_k)}(Q)$/$d^{(n_k)}_*(Q)$ of $Q\subset \Z$ along $(n_k)$ is given by
\begin{equation}\label{eq:density-along}
d^*_{(n_k)}(Q)=\limsup_{k\to\infty}\frac{\left|Q\cap [0,n_k)\right|}{n_k}\text{ and }d^{(n_k)}_*(Q)=\liminf_{k\to\infty}\frac{\left|Q\cap [0,n_k)\right|}{n_k}.
\end{equation}
By $d^*(S)/d_*(S)$ we denote the upper/lower asymptotic density of $S\subseteq\N$ along the sequence $n_k=k$ for every $k\in\N$.

The \emph{Besicovitch pseudometric along $(n_k)_{k=1}^\infty$} on $(X,T)$ is given for $x,x'\in X$ by
	\begin{equation}\label{def:D_B-along}
	    D^{(n_k)}_B(x,x')=\limsup_{k\to\infty}\frac{1}{n_k}\sum_{j=0}^{n_k-1}\rho(T^j(x),T^j(x')).
	\end{equation}	

Furthermore, we introduce a variant of the pseudometric  used in \cite[Lemma 2]{KLO2}; however, 
here it is based on the upper density along $(n_k)_{k=1}^{\infty}$ instead of $d^*$. Given a topological dynamical system $(X,T)$, a sequence $(n_k)_{k=1}^{\infty}$, and $x,y\in X$ we set
\begin{equation}\label{defn:equivalent-pseudometric-along-subseq}
        \tilde{D}_{B}^{(n_k)}(x,y)=\inf\left\{ \delta > 0 \mid d^*_{(n_k)}\left(\{ j\in \N \mid \rho(T^j(x),T^j(y))\geq \delta \}\right)< \delta   \right\}.
\end{equation} 
An obvious modification of the proof of \cite[Lemma 2]{KLO2} yields the following useful result, which we state for the sake of completeness.
 
\begin{lem}\label{lem:equiv-D_B-along}
 If $(n_k)_{k=1}^\infty\subseteq\N$ is a strictly increasing sequence, then pseudometrics $D_B^{(n_k)}$ and $\tilde{D}_B^{(n_k)}$ are uniformly equivalent on $(X,T)$.
\end{lem}

The proof of our next result follows the same lines as the proofs of \cite[Lemma 6]{KLO2} and \cite[Theorem 15]{KLO2}. The obvious changes must take into account that all the measures in Theorem \ref{thm:Besicovitch-quasi-genericity-and-ergodicity} are generated along the subsequence $(n_k)$, while the statements in \cite{KLO2} use $n_k=k$ for every $k\ge 1$. More precisely, one needs to use  Lemma \ref{lem:equiv-D_B-along} in place of \cite[Lemma 2]{KLO2} and an  appropriate modification  of 
\cite[Theorem 14]{KLO2}.

\begin{thm}\label{thm:Besicovitch-quasi-genericity-and-ergodicity}
Let $(X,T)$ be a topological dynamical system and $(n_k)_{k=1}^\infty\subseteq\N$ be strictly increasing. 
If for every $m\ge 1$ the point $x_m\in X$ is quasi-generic for $\mu_m\in\MT(X)$ along  $(n_k)_{k=1}^\infty$   
and $x\in X$ is such that 
\[
\lim_{m\to\infty}
D^{(n_k)}_B(x,x_m)=0, 
\]
then $x$ is quasi-generic along  $(n_k)_{k=1}^\infty$ for some $\mu\in\MT(X)$, where $\mu_m$ converges to $\mu$ in the weak$^*$ topology
as $m\to\infty$.
Furthermore, if $\mu_m$ is ergodic for every $m\ge 1$, then so is $\mu$.
\end{thm}

Consider a finite discrete space $A$ (an alphabet). Let $A^\Z$ be the topological space being a product of a bi-infinite sequence of copies of $A$. Let $\sigma\colon A^\Z\to A^\Z$ be the shift transformation, where $\sigma(x)_i=x_{i+1}$ for $i\in\Z$. The $\dbar_{(n_k)}$-pseudometric along a strictly increasing sequence $(n_k)$ on $A^\Z$ measures the upper asymptotic density along $(n_k)$ of the set of nonnegative coordinates at which two $A$-valued biinfinite sequences differ, that is, for $x=(x_j)_{j\in\Z}$ and $y=(y_j)_{j\in\Z}$ we set
\[\dbar_{(n_k)}(x,y)=d^*_{(n_k)}\{j\in\N_0:x_j\neq y_j\}=\limsup_{k\to \infty}\frac{1}{n_k}|\{0\leq j<n_k:x_j\neq y_j\}|.\]
Let $\rho$ be any metric on $A^\Z$ compatible with the product topology. It turns out that the Besicovitch pseudometric along $(n_k)$ generated by $\rho$ on $A^\Z$ is uniformly equivalent with the $\dbar_{(n_k)}$-pseudometric.  A proof follows the same lines as the proof of \cite[Theorem 4]{KLO2}, which yields the same conclusion in case that $n_k=k$.

\begin{lem}\label{lem:dbar-D_B-equiv}
If 
$A$ is a finite discrete space, $\rho$ is a metric on $A^\Z$ compatible with the product topology, and $(n_k)_{k=1}^\infty\subseteq\N$ is a strictly increasing sequence, then pseudometrics: $D_B^{(n_k)}$ given by \eqref{def:D_B-along} and $\dbar_{(n_k)}$ are uniformly equivalent on $A^\Z$.    
\end{lem}
The next Lemma has not appeared in \cite{KLO2} in any form, hence we will present its proof in detail.
\begin{lem}\label{lem:prox_Av_n} Let $f\in C(X)$.
For every $\eps>0$ there is $\delta>0$ such that if $x,x'\in X$ satisfy
$D_B(x,x')<\delta$, then there is $N\ge 1$ such that for every $n\ge N$ and $\xi\in\bS^1$ we have
\[
\left|\Av_n[f,\bar{\xi}](x)-\Av_n[f,\bar{\xi}](x')\right|\le\eps.
\]
\end{lem}
\begin{proof}
Fix $\eps>0$ and $f\in C(X)$. We pick $\alpha>0$ such that if $z,z'\in X$ and $\rho(z,z')<\alpha$, then $\lvert f(z)-f(z')\rvert<{\eps}/{3}$. Set $\beta=\min\{\alpha,\eps/(4\max\{1,\lVert f\rVert_{\infty}\})\}$. Next, we use Lemma \ref{lem:equiv-D_B-along} for $n_k=k$ to  choose $\delta>0$ so small that $D_B(x,x')<\delta$ implies
\[
d_*\left(\left\{ j\ge 0 : \rho\left( T^j(x), T^j(x') \right)<\beta\right\}\right)\ge 1-\beta.
\]
It follows that there is $N\ge 1$ such that for every $n\ge N$ we have 
    \begin{equation*}
    \frac1n\left\lvert\left\{0\le j<n: \rho\left( T^j(x), T^j(x') \right)<\alpha\right\}\right\rvert\geq 1-\dfrac{\eps}{3\max\{1,\lVert f\rVert_{\infty}\}}.
    \end{equation*}   
Take $\xi\in\bS^1$. Fix $n\ge N$ and define 
\[
I(n)=\left\{0\le j<n: \rho\left( T^j(x), T^j(x') \right)<\alpha\right\}.
\]
Note that $j\in I(n)$ implies $|f( T^j(x) )-  f( T^j(x') )|<\eps/3$. 
Let $J(n)$ be the complement of $I(n)$ in $\{0,1,\ldots,n-1\}$. Hence, 
\begin{equation}
    \label{ineq:jn} \dfrac{1}{n}\lvert J(n) \rvert \lVert f\rVert_{\infty} < \dfrac{\eps}{3}\frac{\lVert f\rVert_{\infty}}{\max\{1,\lVert f\rVert_{\infty}\}} 
    \le \dfrac{\eps}{3}.
\end{equation}
Furthermore, for $j\in J(n)$ we have 
the trivial bound
\[
\left|f\left( T^j(x) \right)-  f\left( T^j(x') \right)\right|\le 2\lVert f\rVert_\infty.
\]
Observe that
\begin{align*}
\lvert\Av_n[f,\bar{\xi}](x) & -\Av_n[f,\bar{\xi}](x')\rvert 
    = \left| \dfrac{1}{n}\sum_{j=0}^{n-1} f\left( T^j(x) \right)\xi^{-j} - \dfrac{1}{n}\sum_{j=0}^{n-1} f\left( T^j(x') \right)\xi^{-j}\right|   \\ &\le 
         \dfrac{1}{n}\sum_{j=0}^{n-1} \left|f\left( T^j(x) \right)-  f\left( T^j(x') \right)\right|\left|\xi^{-j}\right|\\
         &\le 
         \dfrac{1}{n}\left(\sum_{j\in I(n)} |f( T^j(x))-  f( T^j(x') )|+\sum_{j\in J(n)} |f( T^j(x) )-  f( T^j(x'))|\right)\\
         &\le \dfrac{1}{n}\left(\dfrac{\eps}{3}\lvert I(n)\rvert+  2 \lvert J(n) \rvert \lVert f\rVert_{\infty}\right)\le \dfrac{\eps}{3}+  2 \dfrac{\eps}{3}=\eps. 
\end{align*}
This completes the proof.
\end{proof}

\section{A new metric on the space of invariant measures}\label{sec:rho-bar}

We study the metric $\bar{\rho}$ on the space $\MT(X)$. We show that $\bar{\rho}$ is a metric inducing a topology stronger than the weak$^*$ topology. The metric $\bar{\rho}$  generalises Ornstein's $\dbarm$-metric on the space of shift-invariant measures (see \cite{Ornstein1974,Shields}). We denote Ornstein's $\dbar$-metric by $\dbarm$ to avoid confusion with the pseudometric $\dbar$ on $A^\Z$ as discussed in \S \ref{sec:Besicovitch}. We recall its definition below. To our best knowledge, the metric $\bar{\rho}$ was not studied before. Nevertheless, some generalisations of Ornstein's $\dbarm$-distance similar to $\bar{\rho}$ have been considered before. R. M. Gray, D. L. Neuhoff and P. C. Shields in \cite{GNS} used one of these generalisations in applications to coding theory. The results of \cite{GNS} are phrased in the language of stochastic processes, rather than in that of homeomorphisms and their invariant measures. In particular, $\dbarm$ is viewed there as a distance between finite valued discrete-time stochastic processes, and its generalisation in \cite{GNS} is a metric for discrete-time stochastic processes having values in separable metric spaces. Another generalisation was provided by Schwarz \cite{Schwarz}, who also used the framework of discrete time processes to prove some results of D. S. Ornstein without invoking an approximation by discrete partitions. 

Assume that $(X,T)$ and $(Y,S)$ are topological dynamical systems. By $\pi^X$ (respectively, $\pi^Y$) we denote the projection of $X\times Y$ onto $X$ (respectively, $Y$). If $\lambda$ is a measure on $X\times Y$, then we define the push-forward measure $\pi^Z_*(\lambda)$ on $Z=X$ (respectively, $Y$) by $\pi^Z_*(\lambda)(B)=\lambda((\pi^Z)^{-1}(B))$. A \emph{joining} of $\mu\in \MT(X)$ and $\nu \in \mathcal{M}_S(Y)$ is a $T\times S$-invariant measure $\lambda$ on $X\times Y$ such that $\pi^X_*(\lambda)=\mu$ and $\pi^Y_*(\lambda)=\nu$. Let $J(\mu,\nu)$ stand for the set of all joinings of $\mu$ and $\nu$. This is a nonempty, closed, and convex subset of $\mathcal{M}_{T\times S}(X\times Y)$. Let $F\colon X\times Y\to\C$ be a continuous function. We define
\begin{equation}\label{def:Fbar}
\bar{F}(\mu,\nu)=\inf_{\lambda\in J(\mu,\nu)}\int_{X\times Y} F(x,y) \text{ d}\lambda(x,y).    
\end{equation}

Note that the map $J(\mu,\nu)\ni\lambda\mapsto \int_{X\times Y} F(x,y) \text{ d}\lambda(x,y)\in\C$ is convex and continuous. Hence, the infimum in \eqref{def:Fbar} is attained at some extreme point of $
J(\mu,\nu)$. Furthermore, if $\mu$ and $\nu$ are both ergodic, then the set $J^e(\mu,\nu)$ of extreme points of $J(\mu,\nu)$ consists of ergodic $T\times S$-invariant measures.

Plugging $(Y,S)=(X,T)$ and $F=\rho$ (where $\rho$ is, as usual, a compatible metric on $X$) into \eqref{def:Fbar} we obtain a function $\bar{\rho}$ given for $\mu,\nu\in\MT(X)$ by 
\begin{equation}\label{def:rho-bar}
\bar{\rho}(\mu,\nu)=\inf_{\lambda\in J(\mu,\nu)}\int_{X\times X} \rho(x,y) \text{ d}\lambda(x,y). 
\end{equation}
It is easy to see (cf. the proof of \cite[Theorem 21.1.2]{Garling}) that $\bar{\rho}$ is a metric on $\MT(X)$ (the proof of the triangle inequality hinges on the gluing lemma \cite[Theorem 16.11.1]{Garling}).
The topology induced by $\bar{\rho}$ on $\MT(X)$ is stronger than the weak$^*$ topology on $\MT(X)$, because for every $\mu,\nu\in\MT(X)$ we have 
$W_1(\mu,\nu)\le \bar{\rho}(\mu,\nu)$, where $W_1$ stands for the Wasserstein metric inducing the weak$^*$ topology (see \cite{Garling}).  We record this fact as a remark for further reference.
\begin{rem}\label{rem:weakstar}
If  $\mu\in\MT(X)$ and $(\mu_n)_{n\ge 1}\subseteq \MT(X)$ is a  sequence such that $\bar{\rho}(\mu_n,\mu)\to 0$ as $n\to\infty$,
then $\mu_n$ converges to $\mu$ in the weak$^*$ topology.
\end{rem}

We  recall the definition of Ornstein's metric $\dbarm$. Note that Ornstein's d-bar distance $\dbarm$ is usually denoted by $\dbar$ in the literature, but we frequently use the latter symbol for other purposes, so we need to distinguish between $\dbar$ and $\dbarm$.  
Fix an alphabet $A$ and define $d_0\colon A^\Z\times A^\Z\to\{0,1\}$ by
\[
d_0((x_i)_{i\in\Z},(y_i)_{i\in\Z})=\begin{cases}
    1, &\text{ if }x_0\neq y_0,\\
    0, &\text{ otherwise}.
\end{cases}
\]
Then $d_0$ is a continuous function, so we take $(Y,S)=(X,T)=(A^\Z,\sigma)$ and $F=d_0$ in \eqref{def:Fbar} to get a function given for
$\mu,\nu\in\Ms(A^\Z)$ by
\begin{equation}\label{def:dbarm}
\dbarm(\mu,\nu)=\bar{d}_0(\mu,\nu)=\inf_{\lambda\in J(\mu,\nu)}\int_{A^\Z\times A^\Z} d_0(x,y) \text{ d}\lambda(x,y). 
\end{equation}

\begin{thm}\label{thm:gen-vs-point-from-supp} Let $(X,T)$ and $(Y,S)$ be topological dynamical systems and $F\colon X\times Y\to\R$ be continuous. 
Assume that $(n_k)_{k=1}^\infty\subseteq\N$ and $n_k\nearrow \infty$ as $k\to\infty$. Suppose $x\in X$ (respectively, $y\in Y$) is quasi-generic along $(n_k)_{k=1}^\infty$ for $\mu\in\MT(X)$ (respectively, $\nu\in\MS(Y)$). Then  
\begin{equation}\label{ineq:F-bar}
   \bar{F}(\mu,\nu) \le \liminf_{k\to\infty}
\frac{1}{n_k} \sum_{j=0}^{n_k-1} F(T^j(x),S^j(y)).
  \end{equation}
If, in addition, $\mu$ and $\nu$ are ergodic and if $X'\subseteq X$ and $Y'\subseteq Y$ are Borel sets such that $\mu(X')=\nu(Y')=1$, then
there exist Borel sets $X''\subseteq X'$ and $Y''\subseteq Y'$ such that $\mu(X'')=\nu(Y'')=1$ and for every $x''\in X''$ there exists $y''\in Y''$ such that
\begin{equation}\label{eq:good-lim}
\lim_{n\to\infty}\frac1n \sum_{j=0}^{n-1} F(T^j(x''),S^j(y'')) 
 = \bar{F}(\mu,\nu). 
\end{equation}

\end{thm}
\begin{proof}
Suppose $x\in X$ (respectively, $y\in Y$) is quasi-generic along $(n_k)_{k=1}^\infty$ for $\mu\in\MT(X)$ (respectively, $\nu\in\MS(Y)$). Without loss of generality,
we assume
\begin{equation}\label{eq:hat_n_k}
\liminf_{k\to\infty}
\frac{1}{n_k} \sum_{j=0}^{n_k-1} F(T^j(x),S^j(y))=\lim_{k\to \infty}\frac{1}{{n}_k}\sum_{j=0}^{{n}_k-1} F(T^j(x),S^j(y)).\end{equation}
For $k\ge 1$ consider the probability measure
$\delta_k$ on $X\times Y$ given  by 
\begin{equation}\label{eq:delta_k}
        \delta_k = \dfrac{1}{{n}_k}\sum_{j=0}^{{n}_k-1}\delta_{ (T\times S)^j(x,y)}. 
    \end{equation}
  Since $\mathcal{M}(X\times Y)$ is compact, without loss of generality, we assume that $(\delta_k)_{k=1}^\infty$ converges to    $\lambda\in\mathcal{M}_{T \times S} \left( X\times Y \right)$ (otherwise, we pass to a converging subsequence and rename it $(\delta_k)_{k=1}^\infty$). Combining \eqref{eq:hat_n_k} with \eqref{eq:delta_k} we obtain
\begin{equation}\label{ineq:F-liminf}
    \int_{X\times Y} F \text{ d}\lambda=\int_{X\times Y}  F \text{ d}(\lim_{k \rightarrow \infty} \delta_{k})  =\liminf_{k\to\infty}
\frac{1}{n_k} \sum_{j=0}^{n_k-1} F(T^j(x),S^j(y)).
    \end{equation}
Since $\lambda$ is a joining of $\mu$ and $\nu$,  \eqref{def:Fbar} together with \eqref{ineq:F-liminf} yield \eqref{ineq:F-bar}.

Now, assume that $\mu$ and $\nu$ are ergodic. Consider Borel subsets $X'\subseteq X$ and $Y'\subseteq Y$ such that $\mu(X')=
\nu(Y')=1$.   Let $\lambda\in J(\mu,\nu)$ be such that $\bar{F}(\mu,\nu)=\int F\text{ d}\lambda$. Since $J(\mu,\nu)$ is nonempty, convex and compact, the infimum in \eqref{def:Fbar} must be attained at an extreme point. As $\mu$ and $\nu$ are ergodic, the set of extreme points $J^e(\mu,\nu)$ of $J(\mu,\nu)$ consists of ergodic joinings, so $\lambda$ must be ergodic. 

We denote the set of $\lambda$-generic points  by $\Gen(\lambda)$. 
 Let
\begin{equation}\label{def:Lp}
L'=\left((\pi^X)^{-1} (X')\cap (\pi^Y)^{-1} (Y')\right) \cap \Gen(\lambda)\subseteq X'\times Y'.
\end{equation}
Note that  $\lambda\big((\pi^X)^{-1} (X')\big)=\lambda\big((\pi^Y)^{-1}(Y')\big)=1$. 
It follows that $\lambda\left(L' \right) = 1 $.  If  $(x',y')\in L'$, then by the ergodic theorem 
we have
     \[  \lim_{n\rightarrow \infty}\dfrac{1}{n}\sum_{j=0}^{n-1} F\left( T^j(x'), S^j(y') \right)  =\int_{X\times Y} F \text{ d}\lambda. \]
Since $\lambda$ is a regular measure, for every $n\ge 1$ there is a compact set $K_n'\subseteq L'$ with $\lambda(K_n')>1-1/n$. Let
$X_n''=(\pi^X)(K_n')$ and $Y_n''=(\pi^Y)(K_n')$. Clearly, $X_n''\subseteq X'$ and $Y_n''\subseteq Y'$ are compact sets satisfying
$\mu(X_n'')=\lambda\big((\pi^X)^{-1}(X''_n)\big)\ge \lambda(K_n')>1-1/n$ and
$\nu(Y_n'')=\lambda\big((\pi^Y)^{-1}(Y''_n)\big)\ge \lambda(K_n')>1-1/n$. Denote
\[
X''=\bigcup_{n=1}^\infty X_n''
\text{ and }Y''=\bigcup_{n=1}^\infty Y_n''.\]
Now, $X''\subseteq X'$ and $Y''\subseteq Y'$ are $\sigma$-compact, hence Borel measurable sets with full measure. Furthermore, for every $x''\in X''$ there exists $y''\in Y''$ such that \eqref{eq:good-lim} holds.
\end{proof}

Taking $(Y,S)=(X,T)$ and $F=\rho$ (the metric on $X$), we obtain the following as a consequence of Theorem \ref{thm:gen-vs-point-from-supp}.

\begin{cor}\label{cor:dbar-closeness-for-regular-points}
Let $\mu,\nu\in \MT(X)$. 
\begin{enumerate}
\item \label{dbar_i} If $(n_k)_{k=1}^\infty\subseteq\N$ is a strictly increasing sequence, $x\in X$ (respectively, $y\in Y$) is quasi-generic along $(n_k)_{k=1}^\infty$ for $\mu$ (respectively, $\nu$), and
\[\liminf_{k\to\infty}
\frac{1}{n_k} \sum_{j=0}^{n_k-1} \rho(T^j(x),
T^j(y))\le\eps,\]
then $\bar{\rho}(\mu,\nu)\le \eps$.
\item \label{dbar_ii} 
Let $X_\mu\subseteq \Gen(\mu)$ and $X_\nu\subseteq\Gen(\nu)$ be Borel sets such that $\mu(X_\mu)=\nu(X_\nu)=1$. 
If $\bar{\rho}(\mu,\nu)<\eps$ for some $\eps>0$, then there exist Borel sets $X'_\mu\subseteq X_\mu$, and $X'_\nu \subseteq X_\nu$ such that $\mu(X'_\mu)=\nu(X'_\nu)=1$ and for every $x' \in X'_\mu$ there exists $y'\in X'_\nu$ 
with  $D_B(x',y')\le \eps$.
\end{enumerate}
\end{cor}

\begin{lem}
The metric $\bar{\rho}$ is complete.
\end{lem}
\begin{proof} 
Let $(\mu_m)_{m=1}^\infty\subseteq\MT(X)$ be a Cauchy sequence with respect to $\bar{\rho}$. Since $\MT(X)$ with the weak$^*$ topology is compact, passing to a subsequence if necessary, we assume that $(\mu_m)_{m=1}^\infty$ is weak$^*$ convergent and we denote its limit by $\mu$. It remains to show that $\bar{\rho}(\mu_m,\mu)\to 0$ as $m\to \infty$. Fix $\eps>0$. Let $M\in\N$ be such that for every $i,j\ge M$ we have $\bar{\rho}(\mu_i,\mu_j)<\eps$. We claim that for every $m\ge M$ we have $\bar{\rho}(\mu,\mu_m)\le\eps$. To this end we fix $m\ge M$ and for every $\ell>m$ use the definition of $\bar{\rho}$ to find $\lambda_\ell\in J(\mu_m,\mu_\ell)$ such that
\[
\bar{\rho}(\mu_m,\mu_\ell)\le \int_{X\times X}\rho(x,y)\text{ d}\lambda_\ell(x,y)<\eps.
\]
Since the space $\cM_{T\times T}(X\times X)$ is compact, the sequence 
$(\lambda_\ell)_{\ell=m+1}^\infty$ has a convergent subsequence $\lambda\in \cM_{T\times T}(X\times X)$, which is easily seen to be a joining of $\mu_m$ and $\mu$ witnessing
\[
\bar{\rho}(\mu_m,\mu)\le \int_{X\times X}\rho(x,y)\text{ d}\lambda(x,y)\le\eps.\qedhere
\]
\end{proof}
The following proof is a minor modification of the proof of \cite[Thm. 7.8]{Rudolph}. 

\begin{lem}\label{lem:erg-dbar-closed}
If $(\mu_m)_{m=1}^\infty\subseteq\MTe(X)$ and $\mu\in\MT(X)$ are such that $\bar{\rho}(\mu_m,\mu)\to 0$ as $m\to\infty$, then $\mu$ is ergodic.
\end{lem}

\begin{proof}
    Let $(\mu_m)_{m=1}^\infty\subseteq\MTe(X)$ and $\mu\in\MT(X)$ be such that 
    \begin{equation}\label{conv}
        \bar{\rho}(\mu_m,\mu)\to 0\text{ as }m\to\infty. 
    \end{equation}Fix $\eps>0$ and take $m$ large enough to guarantee that $\bar{\rho}(\mu_m,\mu)<\eps^2$. Let $\Lambda\in J(\mu_m,\mu)$ be such that
\[
    \int_{X\times X} \rho(x,y)\text{ d}\Lambda(x,y)<\eps^2.      
  \]
    Let $\hat\Lambda\in \cM(\cM( X\times X)$ be the ergodic decomposition of $\Lambda$, that is
        \begin{equation}\label{ineq:eps-sq}
    \int_{X\times X} \rho(x,y)\text{ d}\Lambda(x,y)=\int_{\cM^e_{T\times T}(X\times X)}\int_{X\times X} \rho(x,y)\text{ d}\lambda(x,y) \text{ d}\hat\Lambda(\lambda)<\eps^2.
    \end{equation}
Markov's inequality and \eqref{ineq:eps-sq} imply that
\begin{equation}\label{ineq:eps-sq-consequence}
\hat\Lambda\left(\left\{\lambda\in\cM_{T\times T}^e(X\times X): \int_{X\times X} \rho(x,y)\text{ d}\lambda(x,y) \ge \eps\right\}\right)<\eps.    
\end{equation}
    Since $\mu_m$ is ergodic, we see that $\hat\Lambda$-a.e.  $\lambda$ is an ergodic $T\times T$-invariant measure having $\mu_m$ as the first marginal. 
Let $\pi\colon X\times X\to X$ be the projection on the second coordinate. Write $\pi_*\colon\cM(X\times X)\to\cM(X)$ and $\pi_{**}\colon\cM(\cM(X\times X))\to\cM(\cM(X))$ for the pushforward maps. Note that using the formula for the integral of the pushforward measures for every $f\in C(X)$ we have 
\begin{multline*}
\int_X f\text{ d}\mu=\int_{X^2} f\circ\pi \text{ d}\Lambda =\int_{\cM^e_{T\times T}(X^2)} \int_{X^2} f\circ\pi \text{ d}\lambda\text{ d}\hat\Lambda(\lambda)%
=\\=\int_{\MTe(X)} \int_X f \text{ d}\pi_*(\lambda)\text{ d}\pi_{**}\hat\Lambda(\pi_*(\lambda)).
\end{multline*}

In other words, $\hat\zeta=\pi_{**}\hat\Lambda$ is the ergodic decomposition of $\mu$.
We claim that
\begin{equation}\label{claim:zeta}
\hat{\zeta}\left(\{\zeta\in\MTe(X):\bar{\rho}(\mu_m,\zeta)\ge\eps\}\right)< \eps.
\end{equation}
Indeed, if $\zeta\in\MTe(X)$, then $\bar{\rho}(\mu_m,\zeta)\ge \eps$ means that for every $\eta\in J^e(\mu_m,\zeta)$ 
we have
\[
\int_{X\times X}\rho(x,y)\text{ d}\eta\ge \eps.
\]
By definition
\[
\hat{\zeta}\left(\{\zeta\in\MTe(X):\bar{\rho}(\mu_m,\zeta)\ge\eps\}\right)=\hat\Lambda\left(\pi^{-1}_{*}\left(\{\zeta\in\MTe(X):\bar{\rho}(\mu_m,\zeta)\ge\eps\}\right)\right).
\]
But $\hat\Lambda$-a.e. $\lambda$ in $\pi^{-1}_*(\{\zeta\})$ belongs to $J^e( \mu_m,\zeta)$, so \eqref{claim:zeta} holds by \eqref{ineq:eps-sq-consequence}. Clearly $\bar\rho(\mu_m,\zeta)\geq \bar\rho(\mu,\zeta)-\bar\rho(\mu,\mu_m)\geq \bar\rho(\mu,\zeta)-\varepsilon^2$. Therefore
\eqref{claim:zeta}
imply that
\[
\hat{\zeta}\left(\{\zeta\in\MTe(X):\bar{\rho}(\mu,\zeta)\ge\eps+\eps^2\}\right)< \eps.
\]
and therefore $\bar{\rho}(\mu,\zeta)=0$ for $\hat\zeta$ a.e. $\zeta$. But $\bar{\rho}$ is a metric on $\MT$, so 
we conclude $\hat{\zeta}$ must be a Dirac measure concentrated on a single point $\mu$, so $\mu$ has trivial ergodic decomposition, hence it is ergodic.
\end{proof}

It is convenient to reformulate Corollary \ref{cor:dbar-closeness-for-regular-points} in the following way.

\begin{cor} \label{cor:char-rho-bar-conv}   
Let $(\mu_m)_{m=1}^\infty\subseteq\MTe(X)$.  
The following conditions are equivalent:
\begin{enumerate}
    \item \label{cond:char-rho-bar-conv-i}
    There is $\mu\in\MT(X)$ such that $\bar{\rho}(\mu_m,\mu)\to 0$ as $m\to\infty$.
    \item \label{cond:char-rho-bar-conv-ii} There exists 
    $(x_m)_{m=1}^\infty\subseteq X$ and $x\in X$ such that for every $m\ge 1$ we have $x_m\in\Gen(\mu_m)$ and 
    \[
    \lim_{m\to\infty}D_B(x,x_m)=0.
    \]
    \item\label{cond:char-rho-bar-conv-iii} There exist a strictly increasing sequence $(n_k)\subseteq\N$, 
    a sequence $(y_m)_{m=1}^\infty\subseteq X$ such that for every $m\ge 1$ we have $y_m$  is quasi-generic for $\mu_m$ along $(n_k)$ and $y\in X$ such that
    \[
    \lim_{m\to\infty}D_B^{(n_k)}(y,y_m)=0.
    \]    
\end{enumerate}
If one, hence every, condition above holds, then $\mu$ is ergodic, $x\in\Gen(\mu)$, and $y$  is quasi-generic for $\mu$ along $(n_k)$.
In addition, we can guarantee that for every $m\ge 1$ it holds
    \[
D_B(x,x_m)=\lim_{n\to\infty}\frac1n\sum_{j=0}^{n-1}\rho(T^j(x),T^j(x_m)) = \bar{\rho}(\mu_m,\mu).
    \]
\end{cor}

\begin{proof} 
\eqref{cond:char-rho-bar-conv-i}$\implies$\eqref{cond:char-rho-bar-conv-ii} By Lemma \ref{lem:erg-dbar-closed}, $\mu$ is ergodic. For each $m\ge 1$ we use Corollary \ref{cor:dbar-closeness-for-regular-points}\eqref{dbar_ii} to get a Borel sets $X^{(m)}\subseteq\Gen(\mu)$ and $X_m\subseteq\Gen(\mu_m)$  with $\mu(X^{(m)})=\mu_m(X_m)=1$  
such that for every $x'\in X^{(m)}$ there is $x'_m\in X_m$ satisfying
\begin{equation}\label{db-eq}
D_B(x',x_m') = \lim_{n\rightarrow\infty}\dfrac{1}{n}\sum_{j=0}^{n-1}\rho\left(T^j(x'),T^j(x_m')\right)
= \bar\rho(\mu,\mu_m).
\end{equation}
Take $X'=\bigcap_{m=1}^\infty X^{(m)}$. 
Every $x\in X'$ is a $D_B$-limit of a sequence $(x_m)_{m=1}^\infty\subseteq X$ such that for every $m\ge 1$ we have $x_m\in\Gen(\mu_m)$. 

\eqref{cond:char-rho-bar-conv-ii}$\implies$\eqref{cond:char-rho-bar-conv-iii} This is obvious (we take $(n_k)$ such that $n_k=k$ for every $k\ge 1$).

\eqref{cond:char-rho-bar-conv-iii}$\implies$\eqref{cond:char-rho-bar-conv-i} 
Use Corollary \ref{cor:dbar-closeness-for-regular-points}\eqref{dbar_i} to get $\bar{\rho}(\mu_m,\mu)\to 0$ as $m\to\infty$.
\end{proof}

\begin{lem}\label{lem:equivalence-of-metrics-on-symbolic}
    If $A$ is finite and $\rho$ is any metric compatible with the topology on $A^\Z$, then the metrics $\dbarm$ and $\bar{\rho}$ are uniformly equivalent on $\Ms(A^\Z)$.
\end{lem}
\begin{proof} Since $d_0\le \rho$ on $A^\Z\times A^\Z$, we have that
$\dbarm\le \bar{\rho}$  on $\Ms(A^\Z)\times \Ms(A^\Z)$.

By 
Lemma~\ref{lem:dbar-D_B-equiv} the pseudometrics $D_B$ and $\dbar$ are uniformly equivalent on $A^\Z$. Therefore, for every $\eps>0$ there is $\delta>0$ such that for every $x=(x_i)_{i\in\Z}$ and $y=(y_i)_{i\in\Z}$ in $A^\Z$ with $\dbar(x,y)\le \delta$ we have $D_B(x,y)<\eps$.

Assume that $\mu,\nu\in\Mse(A^\Z)$ are such that $\dbarm(\mu,\nu)\le\delta$. Applying 
Corollary \ref{cor:dbar-closeness-for-regular-points}\eqref{dbar_ii}

we find $x'=(x'_i)_{i\in\Z}\in\Gen(\mu)$ and $y'=(y'_i)_{i\in\Z}\in\Gen(\nu)$ 
such that $\dbar(x',y')\le \delta$. Hence, $D_B(x',y')\le \eps$, so applying Corollary \ref{cor:dbar-closeness-for-regular-points}\eqref{dbar_i}, we see that $\bar{\rho}(\mu,\nu)\le\eps$. 
Hence, $\dbarm(\mu,\nu)\le\delta$ implies $\bar{\rho}(\mu,\nu)\le\eps$ on $\Mse(A^\Z)$. 
In order to get the same on $\Ms(A^\Z)$ one needs to consider ergodic decompositions of joinings between $\mu$ and $\nu$. As this is pretty standard, we leave it to the reader (see \cite[Lemma 13 \& Lemma 14]{KKK2}).
\end{proof}

As a direct consequence of Lemma~\ref{lem:dbar-D_B-equiv}, Corollary \ref{cor:char-rho-bar-conv} and Lemma \ref{lem:equivalence-of-metrics-on-symbolic} we obtain the following:

\begin{cor} \label{cor:char-dbarm-conv}   
Let $A$ be an alphabet and $(\mu_m)_{m=1}^\infty\subseteq\Mse(A^\Z)$. The following conditions are equivalent:
\begin{enumerate}
    \item There is $\mu\in\Ms(A^\Z)$ such that $\dbarm(\mu_m,\mu)\to 0$ as $m\to\infty$.
    \item There exist $(x_m)_{m=1}^\infty\subseteq A^\Z$ and $x\in A^\Z$ such that for every $m\ge 1$ we have $x_m\in\Gen(\mu_m)$ and 
    \[
    \lim_{m\to\infty}\dbar(x,x_m)=0.
    \]
    \item There exist a strictly increasing sequence $(n_k)\subseteq\N$, 
    a sequence $(y_m)_{m=1}^\infty\subseteq A^\Z$ such that for every $m\ge 1$ we have $y_m$  which is quasi-generic for $\mu_m$ along $(n_k)$, and $y\in A^\Z$ such that
    \[
    \lim_{m\to\infty}\dbar_{(n_k)}(y,y_m)=0.
    \]    
\end{enumerate}
If one, hence every, condition above holds, then $\mu$ is ergodic, $x\in\Gen(\mu)$, and $y$  is quasi-generic for $\mu$ along $(n_k)$
In addition, we can guarantee that for every $m\ge 1$ it holds
    \[
\dbar(x,x_m)=\lim_{n\to\infty}\frac1n\left|\{0\le j< n: \sigma^j(x)_0\neq\sigma^j(x_m)_0\}\right|.
    \]
\end{cor}

\section{On sequences of generic points converging with respect to the Besicovitch pseudometric}

Theorem \ref{thm:Besicovitch-quasi-genericity-and-ergodicity} says that for every strictly increasing $(n_k)\subseteq\N$ and every $D_B$-limit of a sequence of generic (along $(n_k)$) points for ergodic measures must be a generic (along $(n_k)$) point for an ergodic measure. By Corollary \ref{cor:char-rho-bar-conv} we can rephrase Theorem \ref{thm:Besicovitch-quasi-genericity-and-ergodicity} as the closedness of the set of ergodic measures in the $\bar{\rho}$ metric. Here, we gather some similar results that allow us to describe spectrum of $\bar{\rho}$-limit measures. 

\begin{lem}\label{lem:Freq_WWgen} Let $(\mu_m)_{m=1}^\infty\subseteq\MT(X)$. 
Assume that for each $m\ge 1$ we have a Fourier--Bohr $\mu_m$-generic point $x_m\in X$. If $\mu\in\MT(X)$ and $x\in X$ is a Fourier--Bohr $\mu$-generic point such that $D_B(x,x_m) \to 0$ as $m\to\infty$, 

then 
\begin{equation}\label{eq:liminf_freq}
\Freq(x)\subseteq \liminf_{m\to\infty}\Freq(x_m):=\bigcup_{M=1}^\infty\bigcap_{m=M}^\infty \Freq(x_m).    
\end{equation}
\end{lem}
\begin{proof} 
Let $f\in C(X)$. Fix $\xi\in\bS^1$. Our assumption means 
that 
\begin{align}\label{eq:av_x_m_lim}
    \lim_{n\to\infty}\Av_n[f,\bar{\xi}](x_m)&=\Av[f,\bar{\xi}](x_m)\text{ for every }m\ge 1,\\
    \lim_{n\to\infty}\Av_n[f,\bar{\xi}](x)&=\Av[f,\bar{\xi}](x) \label{eq:av_x_lim}.
\end{align}
First, we prove that 
\begin{equation}\label{eq:goal}
    \lim_{m\to\infty}\Av[f,\bar{\xi}](x_m)=\Av[f,\bar{\xi}](x).
\end{equation}
Fix $\eps>0$. Using Lemma \ref{lem:prox_Av_n}, choose $\delta>0$ for $\eps/3$ and $f\in C(X)$. 
Take $m$ sufficiently large to guarantee $D_B(x,x_m)<\delta$. 
Now, Lemma \ref{lem:prox_Av_n}  implies that
for all but finitely many $n$'s we have
\begin{equation}\label{ineq:Av_n_1}
\left|\Av_n[f,\bar{\xi}](x)-\Av_n[f,\bar{\xi}](x_m)\right|<\eps/3.
\end{equation}
On the other hand, \eqref{eq:av_x_m_lim} and \eqref{eq:av_x_lim} imply  that for all $n$'s large enough it also holds
\begin{align}
\label{ineq:Av_x_m_2}
\left|\Av_n[f,\bar{\xi}](x_m)-\Av[f,\bar{\xi}](x_m)\right|<\eps/3, \\   
\label{ineq:Av_x_2}
\left|\Av_n[f,\bar{\xi}](x)-\Av[f,\bar{\xi}](x)\right|<\eps/3.
\end{align}
Combining \eqref{ineq:Av_n_1}, \eqref{ineq:Av_x_m_2} and \eqref{ineq:Av_x_2} 
we get
\begin{equation}\label{ineq:Av_n_3}
\left|\Av[f,\bar{\xi}](x)-\Av[f,\bar{\xi}](x_m)\right|<\eps.
\end{equation}
for all sufficiently large $m$'s, which proves \eqref{eq:goal}. 

Now, if $\xi\notin \Freq(x_m)$ for infinitely many $m$'s, then using \eqref{eq:goal} we see that for every $f\in\ C(X)$ we have $\Av[f,\bar{\xi}](x)=0$. 
This means that $\xi\notin\Freq(x)$ and \eqref{eq:liminf_freq} holds. 
\end{proof}

\begin{lem}\label{lem:mixing}
Suppose that $(x_m)_{m=1}^\infty\subseteq X$ and for each $m\ge 1$ the point $x_m$ is $\mu_m$-generic 
for a mixing $\mu_m\in\MTe(X)$. If $x\in X$ is such that
\[
\lim_{m\to\infty} D_B(x,x_m) =0,
\]
then $x$ is a generic point for a mixing $\mu\in\MT(X)$.
\end{lem}
\begin{proof} We need to consider various $L^2$-spaces: $L^2(\mu_m)$ for $m\ge 1$ and $L^2(\mu)$. To simplify the notation we denote every Koopman operator  acting on these spaces as $U_T$. We denote the scalar product on $L^2(\mu_m)$ by $\langle\cdot,\cdot\rangle_m$ and write $\langle\cdot,\cdot\rangle$ for the scalar product on $L^2(\mu)$.  Since continuous functions are dense in $L^2(\mu)$, 

it follows from \cite[Exercise 2.7.5]{EW} that it is enough to check that for every $f\in C(X)$ we have
\[\lim_{n\to\infty}\langle U_T^n(f),f\rangle=\langle f,1\rangle\cdot \langle 1,f\rangle,
\]
knowing that for every $m\ge 1$ it holds
\[\lim_{n\to\infty}\langle U_T^n(f),f\rangle_m=\langle f,1\rangle_m\cdot \langle 1,f\rangle_m.
\]
Fix $f\in C(X)$ and $\eps>0$.  Let $\Psi_m(f)=\langle f,1\rangle_m\cdot \langle 1,f\rangle_m$ and $\Psi(f)=\langle f,1\rangle\cdot \langle 1,f\rangle$. Given $n$ and $m\ge 1$ we clearly have
\begin{multline*}
\left|\langle U_T^n(f),f\rangle-\Psi(f)\right|\\ \le \left|\langle U_T^n(f),f\rangle-\langle U_T^n(f),f\rangle_m\right|+\left|\langle U_T^n(f),f\rangle_m-\Psi_m(f)\right|+ 
\left|\Psi_m(f)-\Psi(f)\right|.
\end{multline*}

Our assumption reads that for every $m\ge 1$ and $\eps>0$ there is $N=N(\eps/3,m,f)$ such that for every $n\ge N$ we have $\left|\langle U_T^n(f),f\rangle_m-\Psi_m(f)\right|<\eps/3$. Also, there is $M=M(\eps/3,f)$ such that for every $m\ge M$ we have $\left|\Psi_m(f)-\Psi(f)\right|<\eps/3$.

For every $n\in\N$ and $m\ge 1$ we write $y_n=T^n(x)$ and $y^{(m)}_n=T^n(x_{m})$. With this notation, using that $x$ and $x_m$ are generic points, we have
\begin{gather*}
    \langle U_T^n(f),f\rangle =\int f\circ T^n \cdot \bar f \text{ d}\mu=\lim_{k\to\infty}\frac1k\sum_{j=0}^{k-1}f(T^j(y_n))\overline{f(T^j(x))},\\
    \langle U_T^n(f),f\rangle_m=\int f\circ T^n \cdot \bar f \text{ d}\mu_m=\lim_{k\to\infty}\frac1k\sum_{j=0}^{k-1}f(T^j(y^{(m)}_n))\overline{f(T^j(x_m))}.
\end{gather*}

Therefore, it remains to show that for every $n\ge 1$ and for every sufficiently large $m\ge 1$ there is $K\ge 1$ such that for every $k\ge K$ it holds
\begin{equation}\label{ineq:mix-goal}
\frac1k\sum_{j=0}^{k-1} \left|f(T^j(y_n))\overline{f(T^j(x))}-f(T^j(y^{(m)}_n))\overline{f(T^j(x_m))}\right|\le \eps/3.
    \end{equation}
Write $A_j$ for the $j$th summand in \eqref{ineq:mix-goal} and note that
\begin{multline}\label{Aj-bound}
A_j:=\left|f(T^j(y_n))\overline{f(T^j(x))}-f(T^j(y^{(m)}_n))\overline{f(T^j(x_m))}\right|\le\\ 
 \left(\left|f(T^j(x))-f(T^j(x_m))\right|
 + \left|f(T^j(y_n))-f(T^j(y^{(m)}_n))\right|\right)\norm{f}_\infty\\\le 4\norm{f}_\infty^2.    
\end{multline}

Let $\beta=\eps/(6\max\{1,\norm{f}_\infty\})$. Since $f$ 
is continuous, there is $\alpha>0$ such that if $x,x'\in X$ satisfy 
$\rho(x,x')\le \alpha$, then $\left\lvert f(x)-f(x')\right\rvert\le \beta$.

Now, for every $k\ge 1$ we divide the set  $\{0,1,\ldots,k-1\}$ into three pairwise disjoint subsets: 
\begin{enumerate}
    \item \label{Set_1-mix} The set of indices $j$ such that $k-n\le j < k$. The upper bound for $A_j$ for these $j$'s is $4\norm{f}^2_\infty$.
    \item \label{Set_2-mix} Those indices $j$ in $[0,k-n)$ such that
    \[\rho(T^j(x),T^j(x_m))\le \alpha\quad\text{and}\quad  
    \rho(T^j(y_n),T^j(y_n^{(m)}))\le \alpha.\] For these $j$'s we have  $A_j\le 2\beta \norm{f}_\infty$. Denote the set of these indices by $G^{(m)}(k)$.
    \item \label{Set_3-mix} The indices $j$ in $B^{(m)}(k)=[0,k)\setminus G^{(m)}(k)$. For these $j$'s we have $A_j\le 4\norm{f}^2_\infty$. 
\end{enumerate}
It can be easily observed that
\begin{equation*}
\lim_{m\to\infty}\lim_{k\to\infty}\frac{1}{k}\left| 
G^{(m)}(k)\right|=1 \quad\text{and}\quad
\lim_{m\to\infty}\lim_{k\to\infty}\frac{1}{k}\left|B^{(m)}(k)\right|=0.
\end{equation*} 
Therefore, for all sufficiently large $m$, we can bound the left hand side of \eqref{ineq:mix-goal} by 
\[
\frac{1}{k} \bigg( n\cdot 4\norm{f}^2_\infty + |G^{(m)}(k)|\cdot 2 \beta \norm{f}_\infty + |B^{(m)}(k)|\cdot 4\norm{f}^{2}_\infty\bigg)  < \frac{\eps}{3}. 
\]
Hence, \eqref{ineq:mix-goal} holds.
\end{proof}

Characteristic classes are families of measure-preserving systems introduced in \cite{BKPLR} and are related to an extension of Sarnak's conjecture beyond zero-entropy systems studied in \cite{BKPLR}. A \emph{characteristic class} is the collection $\mathcal{C}$  of all measure-preserving systems that is closed under taking factors and (countable) joinings. Recently, the authors of \cite{BL} proved that for every characteristic class $\mathcal{C}$ the set $\MTe(X)\cap \mathcal{C}$ is closed with respect to $\bar{\rho}$. In particular,  they established in \cite{BL} that if $(x_m)_{m=1}^{\infty}$ is a sequence such that for each $m\ge 1$ the point $x_m$ is generic for some $\mu_m\in \MTe(X)\cap \mathcal{C}$ and $x_m$ converges to $x$ in the Besicovitch pseudometric, then $x$ is generic for $\mu\in \MTe(X)\cap \mathcal{C}$. Since the family of all measure-preserving systems with discrete spectrum forms a characteristic class, the following result is an immediate consequence of \cite{BL}, where $\mathcal{C}$ is the collection of measures with discrete spectrum. It can also be proved directly using techniques similar to those used in the proof of Lemma~\ref{lem:mixing}. We present this proof for completeness. 

\begin{lem}\label{lem:discrete_Besicovitch}
Suppose that $(x_m)_{m=1}^\infty\subseteq X$ and for each $m\ge 1$ the point $x_m$ is $\mu_m$-generic 
where $\mu_m\in\MTe(X)$. Let $\mu\in \MT(X)$. If $x\in X$ is a 
$\mu$-generic point such that
\[
\lim_{m\to\infty} D_B(x,x_m) =0,
\]
and each $\mu_m$ has discrete spectrum, 
then so does $\mu$.
\end{lem}
\begin{proof} 
By Theorem \ref{thm:Besicovitch-quasi-genericity-and-ergodicity} $\mu$ is ergodic. 
By \cite[Proposition 2.11.14]{Tao}, to prove $\mu\in\MTe(X)$ has discrete spectrum, it is enough to prove that  for every $ f\in C(X)$ the set $\{f\circ T^n: n\in\mathbb{Z}\}$ is totally bounded in $L^2(\mu)$. 

Fix $f\in C(X)$. Our assumption reads that for every $m\ge 1$ and $\eps>0$ there is $R=R(\eps/2,m)$ such that the set $\{f\circ T^r: |r|\le R\}$ is $\eps/2$-dense in $\{f\circ T^n: n\in\mathbb{Z}\}$ with respect to the  $L^2(\mu_m)$-norm. Therefore, for every $n\in\mathbb{Z}$ there is  $r$ such that $|r|\le R$ and $\norm{f\circ T^n-f\circ T^r}_{L^2(\mu_m)}\le \eps/2$. We claim that for all sufficiently large $m$'s, 
\begin{equation}\label{ineq:for-large-m}
     \left\lvert
     \norm{f\circ T^p-f\circ T^q}_{L^2(\mu_m)}-\norm{f\circ T^p-f\circ T^q}_{L^2(\mu)}
     \right\rvert \le \eps/2,\quad \forall p,q\in\Z. 
\end{equation} 
This ensures that for sufficiently large $m$ and $R=R(\eps/2,m)$ the set $\{f\circ T^r: |r|\le R\}$ is $\eps$-dense in $\{f\circ T^n: n\in\mathbb{Z}\}$ with respect to the $L^2(\mu)$-norm.   

Fix $p,q\in\Z$. For every $n\in\Z$ and $m\ge 1$, denote by $\bar{x}_n=T^n(x)$ and $x^{(m)}_n=T^n(x_{m})$. With this notation, using that $x$ and $x_m$ are generic points, we have  
 \begin{multline}\label{ineq:generic_discrete_characterisation}
 \left\lvert
     \norm{f\circ T^p-f\circ T^q}^2_{L^2(\mu_m)}-\norm{f\circ T^p-f\circ T^q}^2_{L^2(\mu)}\right\rvert \\
 \le 
 \lim_{k\to\infty}\frac{1}{k}\sum_{j=0}^{k-1}\left\lvert\lvert f(x_{p+j}^{(m)})-f(x_{q+j}^{(m)})\rvert^2 -\lvert f(\bar{x}_{p+j})-f(\bar{x}_{q+j})\rvert^2
     \right\rvert.
     \end{multline}
 For every $k\ge 1$, by partitioning the sets $P(k)=\{p,p+1,\ldots,p+k-1\}$ and $Q(k)=\{q,q+1,\ldots,q+k-1\}$ into three pairwise disjoint subsets, in a similar way as in the proof of Lemma~\ref{lem:mixing}, we can bound the right hand side of the inequality \eqref{ineq:generic_discrete_characterisation} with $\eps/2$ for sufficiently large $m$.
 \end{proof}

\begin{thm}\label{thm:Besicovitch-spectrum}
If $(\mu_k)_{k=1}^\infty\subseteq\MTe(X)$ converges in $\bar\rho$ to 
$\mu\in\MTe(X)$, then
\[
\Spec(X,T,\mu)\subseteq \liminf_{k\to\infty}\Spec(X,T,\mu_k):=\bigcup_{K=1}^\infty\bigcap_{k=K}^\infty \Spec(X,T,\mu_k).
\]

Furthermore, if for each $k\ge 1$ the measure $\mu_k$ has discrete spectrum (is weakly mixing), then so is $\mu$.
\end{thm}
\begin{proof}  
By Theorem \ref{thm:WW-reg-erg}, there exists a Borel set $X_\mu$ with $\mu(X_\mu)=1$  consisting of Wiener--Wintner $\mu$-generic points. Similarly, for each $k\geq 1$
let $X_{\mu_k}$ be the full measure Borel set of Wiener--Wintner $\mu_k$-generic points. 
For each $k\ge 1$ we use Corollary \ref{cor:dbar-closeness-for-regular-points} to find Borel sets $X_\mu^{(k)}\subseteq X_\mu$ 
and $X_k'\subseteq X_{\mu_k}$
such that for every $x'\in X_\mu^{(k)}$ there is $x'_k\in X_k'$ satisfying
\begin{equation}\label{dbxxk}
D_B(x',x_k') 
\leq \bar\rho(\mu,\mu_k).
\end{equation}
Take $X_\mu'=\bigcap_{k=1}^\infty X_\mu^{(k)}$. Hence, $X_\mu'$ is a full-measure Borel set consisting of Wiener--Wintner $\mu$-generic points.
Fix $x\in X'_\mu$ and for every $k\ge1$ take  $x_k\in X_k'$ so that \eqref{dbxxk} holds. Hence,   $D_B(x,x_k)\rightarrow0$ as $k\to \infty$.
Since $x$ and $x_k$ are Wiener--Wintner generic points 
using Lemma \ref{lem:Freq_WWgen} and Lemma \ref{lem:freq=spec} we get
\begin{equation}\label{eq:liminf_spec}
\Spec(X,T,\mu)=\Freq(x)\subseteq \bigcup_{K=1}^\infty\bigcap_{k=K}^\infty \Freq(x_k)=\bigcup_{K=1}^\infty\bigcap_{k=K}^\infty \Spec(X,T,\mu_k).    
\end{equation}
To prove the furthermore part, assume $\mu_k$ is weakly mixing for every $k\geq 1$. Consequently, $\Spec(X,T,\mu_k)=\{1\}$ for every $k\geq 1$. From \eqref{eq:liminf_spec} we  infer that $\Spec(X,T,\mu)=\{1\}$, proving that $\mu$ is weakly mixing. 

The conclusion for the discrete spectrum case follows from  Lemma \ref{lem:discrete_Besicovitch}. \end{proof}

\subsection{Disjointness of measure preserving systems}
By a standard measure preserving system we understand here a system of the form $(X,T,\mu)$, where $(X,T)$ is a topological dynamical system  and $\mu$ is a completion of $T$-invariant Borel probability measure. We refer to such as system as to $\mu$ when $(X,T)$ is irrelevant for our considerations. Usually, measure preserving systems are understood to be given by automorphisms of a standard Lebesgue space, but every measure preserving system given by such an automorphism is isomorphic to the one given by a completion of an invariant measure for a topological dynamical system.  
We say that measure preserving system $(X,T,\mu)$ and $(Y,S,\nu)$ are \emph{disjoint} (for short, we say that $\mu$ and $\nu$ are disjoint) if $\mu\times\nu$ is the only joining of $\mu$ and $\nu$. 
We write $\mu\perp\nu$ to denote that $\mu$ and $\nu$ are disjoint measures.
Given a class of measure preserving systems $\mathscr{C}$ we write
\[
\mathscr{C}_{\perp}=\{
\mu :\forall\nu\in\mathscr{C}\ \mu\perp\nu\}
\}
\]
for the class of all measure preserving system disjoint with every member of $\mathscr{C}$. The class of all measure preserving system disjoint with every ergodic member of $\mathscr{C}$
is denoted by $\mathscr{C}_{\perp}^{e}$. Our next theorem says that
for every topological dynamical system $(X,T)$ the set $\MTe(X)\cap \mathscr{C}_{\perp}^{e}$ is $\bar{\rho}$ closed.

\begin{thm} \label{thm:dbar-disj} Let $(\mu_m)_{m=1}^\infty\subseteq\MTe(X)$  be such that $\mu_m\in\mathscr{C}_{\perp}^e$ for every $m\ge 1$ for some class $\mathscr{C}$ of measure preserving systems. 
    If $\bar{\rho}(\mu_m,\mu)\to 0$ as $m\to\infty$, then
$\mu\in\mathscr{C}_{\perp}^e$.
\end{thm}
\begin{proof}
Let $(\mu_m)_{m=1}^\infty\subseteq\MTe(X)$  and $\mu\in\MT(X)$ be such that $\mu_m\in\mathscr{C}_{\perp}^e$ for every $m\ge 1$ and $\bar{\rho}(\mu_m,\mu)\to 0$ as $k\to\infty$. By Lemma~\ref{lem:erg-dbar-closed} we obtain that $\mu\in\MTe(X)$.
Take any ergodic  $\nu'\in\mathscr{C}$. Let $(Y,S)$ be a strictly ergodic topological model for $\nu'$, that is $(Y,S)$ is a topological dynamical systems such that $\MSe(Y)=\{\nu\}$, where $\nu$ is isomorphic with $\nu'$.  Let $\lambda\in J^e(\mu,\nu)$. We want to show $\lambda=\mu\times\nu$ and it is enough to prove that for  $f\in C(X)$ and $g\in C(Y)$ we have
\begin{equation}\label{eq:prod-goal}
\int f(x) g(y) \text{ d}\lambda(x,y)=\left(\int f\text{ d}\mu\right)\cdot \left(\int g\text{ d}\nu\right).    
\end{equation}
To this end, we note that since $\lambda$ is ergodic and $\mu$ is its marginal, there is a Borel set $X_\lambda\subseteq\Gen(\mu)$ such that $\mu(X_\lambda)=1$ and for every $x\in X_\lambda$ there is $y\in Y$ such that $(x,y)\in\Gen(\lambda)$. Hence, by Theorem~\ref{thm:gen-vs-point-from-supp} (cf. Corollary \ref{cor:dbar-closeness-for-regular-points}\eqref{dbar_ii}) there is a Borel set $X^*\subseteq\Gen(\mu)$ such that if $x\in X^*$, then there is $y\in Y$ such that $(x,y)\in \Gen(\lambda)$ and for every 
$m\ge 1$ there is $x_m\in\Gen(\mu_m)$ with $D_B(x,x_m)=\bar{\rho}(\mu,\mu_m)$. 
For $x\in X$ and $y\in Y$ set $h(x,y)=f(x)g(y)$ and
\[\Av_n[h](x,y)=\frac1n\sum_{j=0}^{n-1}h(T^j(x),S^j(y))=\frac1n\sum_{j=0}^{n-1}f(T^j(x))\cdot g(S^j(y)).\] Given a continuous function $\varphi$ on a compact metric space and a Borel probability measure $\eta$ on the same space we define 
\[
\tilde{\varphi}(\eta)=\int\varphi\text{ d}\eta.
\]
Fix $x\in X^*$ and $y\in Y$ such that $(x,y)\in \Gen(\lambda)$.  Note that since $(Y,S)$ is strictly ergodic we have $y\in \Gen(\nu)$ and so by  $\mu_m\perp\nu$ we have that $(x_m,y)\in\Gen(\mu_m\times\nu)$ for $m\ge 1$. 
Let $\beta=\eps/(6\max\{1,\norm{g}_\infty\})$. Using continuity of $f$ we find $\alpha>0$ such that if $x,x'\in X$ satisfy 
$\rho(x,x')\le \alpha$, then $\left\lvert f(x)-f(x')\right\rvert\le \beta$.
For every $n,m\ge 1$ we have
\begin{multline}\label{ineq:perp-goal}
\left| \Av_n[h](x,y)-\tilde{f}(\mu)\tilde{g}(\nu)\right|
\le \left| \Av_n[h](x,y)-\Av_n[h](x_m,y)\right|+\\
+\left| \Av_n[h](x_m,y)-\tilde{f}(\mu_m)\tilde{g}(\nu)\right|
+\left| \tilde{f}(\mu_m)\tilde{g}(\nu)-\tilde{f}(\mu)\tilde{g}(\nu)\right|.
\end{multline}
Since $h=f\cdot g$ is continuous and $(x_m,y)\in\Gen(\mu_m\times\nu)$, we see that for every $m\ge 1$ we have
\[
\lim_{n\to\infty}\Av_n[h](x_m,y)= \int h \text{ d}(\mu_m\times\nu)=\tilde{f}(\mu_m)\tilde{g}(\nu).
\]
By Remark~\ref{rem:weakstar}, $\bar{\rho}(\mu_m,\mu)\to 0$ as $m\to\infty$ implies weak$^*$ convergence, thus $\tilde{f}(\mu_m)\to\tilde{f}(\mu)$ as $m\to\infty$. It remains to show that $\left| \Av_n[h](x,y)-\Av_n[h](x_m,y)\right|$ can be arbitrarily small for all sufficiently large $m$ and $n$. Note that
\begin{equation}\label{ineq:to-be-bounded}
    \left| \Av_n[h](x,y)-\Av_n[h](x_m,y)\right|\le \norm{g}_{\infty}\frac1n\sum_{j=0}^{n-1}\left| f(T^j(x))-f(T^j(x_m))\right|.
\end{equation}
Now, for every $n\ge 1$ we divide the set  $\{0,1,\ldots,n-1\}$ into two disjoint subsets:
\begin{enumerate}
    \item \label{Set_1-perp} Those indices $j$ in $[0,n)$ such that
    \[\rho(T^j(x),T^j(x_m))\le \alpha.\] For these $j$'s we have  
    \begin{equation}\label{G-bound}
        \norm{g}_\infty\left| f(T^j(x))-f(T^j(x_m))\right|\le \beta \norm{f}_\infty\norm{g}_\infty<\eps/6.
    \end{equation} Denote the set of these indices by $G^{(m)}(n)$.
    \item \label{Set_2-perp} The indices $j$ in $B^{(m)}(n)=[0,n)\setminus G^{(m)}(n)$. For these $j$'s we have 
    \begin{equation}\label{B-bound}
        \norm{g}_\infty\left| f(T^j(x))-f(T^j(x_m))\right|\le 2 \norm{f}_\infty\norm{g}_\infty.
    \end{equation}
\end{enumerate}
Combining \eqref{G-bound} and \eqref{B-bound} we have the following bound for the right hand side of \eqref{ineq:to-be-bounded}: 
\[
\frac{1}{n} \bigg( |G^{(m)}(n)|\cdot \frac{\eps}{6} + |B^{(m)}(n)|\cdot 2\norm{f}^{2}_\infty\norm{g}_\infty\bigg).  \]
Using Lemma \ref{lem:equiv-D_B-along} and $D_B(x_m,x)\to 0$ as $m\to\infty$ it can be easily observed that
\begin{equation*}
\lim_{m\to\infty}\lim_{n\to\infty}\frac{1}{n}\left| 
G^{(m)}(n)\right|=1 \quad\text{and}\quad
\lim_{m\to\infty}\lim_{n\to\infty}\frac{1}{n}\left|B^{(m)}(n)\right|=0.
\end{equation*}

Therefore for all sufficiently large $m$ and $n$ we can bound each term of the right hand side of \eqref{ineq:perp-goal} by $\eps/3$.
Summing up, for every $\eps>0$ and all sufficiently large $n$'s we have $\left| \Av_n[h](x,y)-\tilde{f}(\mu)\tilde{g}(\nu)\right|<\eps$, and consequently
\[
\lim_{n\to \infty}
\Av_n[h](x,y)=
\left(\int f\text{ d}\mu\right)\cdot \left(\int g\text{ d}\nu\right).
\]
On the other hand, we picked $(x,y)$, so that it is a generic point for $\lambda$. Since $h=f\cdot g$ is continuous, we see that
\[
\lim_{n\to \infty} \Av_n[h](x,y)=\int f(x) g(y) \text{ d}\lambda(x,y).
\]
We obtain that \eqref{eq:prod-goal} holds for every $f\in C(X)$ and $g\in C(Y)$, as needed. The proof is completed.
\end{proof}

We summarise the closedness properties of various classes of measures with respect to the $\bar{\rho}$ metric in the following  result. Note that any of the equivalent conditions appearing in Corollary \ref{cor:char-rho-bar-conv} can be taken as the assumption in Theorem \ref{thm:Besicovitch-summary}.  In the statement, we refer to the notion of $K$-(measure-preserving) systems ($K$ is for Kolmogorov). The actual definition is not important here (see \cite[Def. 4.5]{Rudolph}). We base our observation on the characterisation of $K$-systems provided by \cite[Thm. 6.11]{Rudolph}. 

\begin{thm}\label{thm:Besicovitch-summary}
Let $(X,T)$ be a topological dynamical system and $(\mu_m)_{m=1}^\infty\subseteq\MTe(X)$ be such that $\bar{\rho}(\mu_m,
\mu)\to 0$ as $m\to \infty$ for some $
\mu\in\MT(X)$. 
Then the following holds: 
\begin{enumerate}
\item \label{thm:conclusion-1+} 
The measure $\mu$ is ergodic and $\mu_m\to\mu$ as $m\to\infty$ in the weak$^*$ topology.
\item \label{thm:conclusion-2} It holds $\displaystyle h(\mu)\le \liminf_{m\to\infty} h(\mu_m)$. 
\item \label{thm:conclusion-3} We have $\displaystyle\Spec(X,T,\mu)\subseteq \liminf_{k\to\infty}\Spec(X,T,\mu_k)=\bigcup_{K=1}^\infty\bigcap_{k=K}^\infty \Spec(X,T,\mu_k)$.
\item \label{thm:conclusion-4} If $\mu_m$ has (rational) discrete spectrum for every $m\ge 1$, then so does $\mu$.
\item \label{thm:conclusion-5} If $\mu_m$ is totally ergodic/weakly mixing/mixing  for all $m$'s, then so is $\mu$. 
\item \label{thm:conclusion-6}  If $\mu_m$ is a $K$-system for every $m$, then so is $\mu$.
\end{enumerate}
\end{thm}
\begin{proof} As noted above, throughout the proof we use frequently the equivalences  provided by Corollary~\ref{cor:char-rho-bar-conv}. 
By Lemma~\ref{lem:erg-dbar-closed} and Remark~\ref{rem:weakstar} we obtain \eqref{thm:conclusion-1+}. 
We note that \eqref{thm:conclusion-2} is actually a special case of a result obtained in \cite{FK-preprint}. To get it, one needs to combine \cite[Lemma 25]{FK-preprint}, \cite[Theorem 33]{FK-preprint}, and \cite[Theorem 41]{FK-preprint}. 
Theorem \ref{thm:Besicovitch-spectrum} combined with \eqref{thm:conclusion-1+} imply \eqref{thm:conclusion-3} and \eqref{thm:conclusion-4}. To get \eqref{thm:conclusion-5}, we need to combine Theorem \ref{thm:Besicovitch-spectrum} with spectral characterisations of total ergodicity and weak mixing: a measure-preserving system is totally ergodic (respectively, weakly mixing) if and only if $1$ is the only rational point in the spectrum (respectively, the only element of the spectrum). Mixing follows from Lemma \ref{lem:mixing}.

Finally, to prove \eqref{thm:conclusion-6}, we recall \cite[Thm. 6.11]{Rudolph}, which states that an ergodic measure has property $K$ if it is disjoint from every ergodic zero entropy measure preserving system. In our notation, writing $\mathscr{K}$ for the class of $K$-systems and $\mathscr{Z}$ for the class of zero entropy systems, we have $\mathscr{K}\cap \MTe(X)=\mathscr{Z}_{\perp}^e\cap \MTe(X)$. Now, our conclusion follows for Theorem \ref{thm:dbar-disj}.
\end{proof}
 
Our next result is a direct consequence of Lemma \ref{lem:equivalence-of-metrics-on-symbolic} and Theorem \ref{thm:Besicovitch-summary}. 
Note that any of the equivalent conditions appearing in Corollary \ref{cor:char-dbarm-conv} can be taken as the assumption in Corollary \ref{cor:Besicovitch-summary-symbolic}. Observe also that Corollary \ref{cor:Besicovitch-summary-symbolic}\eqref{cond:entropy} has a proof independent of the proof of Theorem \ref{thm:Besicovitch-summary}, see \cite{Shields}. We stated it here only for completeness.

\begin{cor}\label{cor:Besicovitch-summary-symbolic}
Let $A$ be an alphabet. 
If 
there is $\mu\in\Ms(A^\Z)$ such that 
for some sequence $(\mu_m)\subseteq\Mse(A^\Z)$ we have $\dbarm({\blue \mu_m},\mu)\to0$ as $m\to\infty$, then: 
\begin{enumerate}
    \item The measure $\mu$ is ergodic and $\mu_m\to\mu$ as $m\to\infty$ in the weak $^*$ topology.   
    \item\label{cond:entropy} It holds $h(\mu_m)\to h(\mu)$ as $m\to\infty$ 
    \item  $\displaystyle\Spec(X,T,\mu)\subseteq \liminf_{k\to\infty}\Spec(X,T,\mu_k):=\bigcup_{K=1}^\infty\bigcap_{k=K}^\infty \Spec(X,T,\mu_k)$.
   \item  If $\mu_m$ has (rational) discrete spectrum for all $m$'s then so does $\mu$.
   \item \label{mixing} If $\mu_m$ is totally ergodic/weakly mixing/mixing for all $m$'s then so does $\mu$.
   \item \label{K-property}  If $\mu_m$ is a $K$-system for every $m$, then so is $\mu$.

\end{enumerate}
\end{cor}

\section{Applications}

We discuss applications of $\bar{\rho}$-metric. First, we discuss applications that reprove some known results or can be obtained by using only symbolic variants $\dbar$ and $\dbarm$. 

\subsection{Application for \texorpdfstring{$\mathscr{B}$}{B}-free shifts} 
Our first example of a family of symbolic dynamical systems to which we apply Corollary \ref{cor:Besicovitch-summary-symbolic} is the extensively studied class of $\mathscr{B}$-free shifts. These systems are dynamical counterparts of $\mathscr{B}$-free numbers. There are other situations where one can apply our results in the symbolic setting; see \cite{KKK2, CKKK, KKL}. 
We have chosen $\mathscr{B}$-free shifts, as they form a 'natural' (rather than 'constructed merely for illustration') set of examples. We do not prove anything new about $\mathscr{B}$-free shifts; we merely recover known results. Nevertheless, we want to point out that an explicit reference to $\dbarm$-convergence clarifies the reasoning. While some of the original proofs used the $\dbar$-pseudometric (sometimes in disguise), none of these approaches (except \cite{BKPLR}) did so in the full generality---there wer always necessary computations relying on the number-theoretic nature of $\mathscr{B}$-free shifts. Furthermore, the original proofs did not refer to $\dbarm$.

Sets of $\mathscr{B}$-free numbers have been investigated by many mathematicians, such as Behrend, Chowla, Davenport, Erd\H{o}s, etc. Subsequently, Sarnak \cite{Sarnak} pioneered a dynamical viewpoint on $\mathscr{B}$-free sets by examining a symbolic dynamical system induced by square-free integers (also known as the \emph{square-free shift}). This system is formed by taking the closure of the orbit of the characteristic function of the set of square-free integers under the shift operation on $\{0,1\}^\Z$. Sarnak's approach has been later extended to an arbitrary $\mathscr{B}$-free shifts, see \cite{ALR,DSKPL,Avdeeva,CSG,KPLW}. In \cite{BKPLR} the spectrum of $\mathscr{B}$-free shifts is studied.

Let $\mathscr{B}\subseteq \N$. 
The set of $\mathscr{B}$-free integers $\mathscr{F}_{\mathscr{B}}$ is defined by
\begin{equation}\label{dfn:B-free-set}
    \mathscr{F}_{\mathscr{B}}=\Z\setminus\bigcup_{b\in \mathscr{B}}b\Z. 
\end{equation}
The characteristic sequence of $\mathscr{F}_{\mathscr{B}}$, denoted by $\eta_{\mathscr{B}}$ is an element of $\{0,1\}^\Z$ and it is given for $j\in\Z$ by
\begin{equation}\label{dfn:B-free-char-seq}
    \left( \eta_{\mathscr{B}}\right)_j = \begin{cases}
        1, & \text{if } j \in \mathscr{F}_{\mathscr{B}}, \\
        0, & \text{otherwise}.
    \end{cases}
\end{equation}

\begin{defn}
    The $\mathscr{B}$-free shift $X_{\mathscr{B}}$ is the closure of the orbit of $\eta_\mathscr{B}$ in $\{0,1\}^\Z$.
\end{defn}

Let $\mathscr{B}(k)=\{b_1,\ldots,b_k\}$ with $b_i<b_j$ if $i<j$ be the set of the $k$ smallest members of $\mathscr{B}$. 
It is easy to see that  $\eta_{\mathscr{B}(k)}$ is a periodic point in $\{0,1\}^\Z$ and $\bkshift$ is just a periodic orbit. We recall the well-known Davenport-Erd\H{o}s theorem.

\begin{thm}[see \cite{DE}]\label{thm:Davenport-Erdos}
    Let $\eta_{\mathscr{B}(k)}$ and $\eta_\mathscr{B}$ be the characteristic sequences for sets, respectively, $\mathscr{F}_{\mathscr{B}(k)}$ and $\mathscr{F}_{\mathscr{B}}$. Then
    \begin{equation}\label{eq:conv-of-char-seq}
        \lim_{k\rightarrow\infty}\liminf_{n\rightarrow\infty}\dfrac{1}{n}\left|\{ 0\leq j <k  : \left(\eta_{\mathscr{B}(k)}\right)_j \neq \left(\eta_{\mathscr{B}}\right)_j\}\right|=0.
    \end{equation}
\end{thm}

Actually, the proof of the Davenport-Erd\H{o}s theorem yields that there is one common subsequence along which the $\liminf$ in \eqref{eq:conv-of-char-seq} is achieved.
\begin{thm}
\label{thm:Davenport-Erdos+}
    Let $\eta_{\mathscr{B}(k)}$ and $\eta_\mathscr{B}$ be the characteristic sequences for sets, respectively, $\mathscr{F}_{\mathscr{B}(k)}$ and $\mathscr{F}_{\mathscr{B}}$. 
    There exists a strictly increasing sequence $(n_k)_{k=1}^\infty$ such that
    \begin{equation}\label{eq:conv-of-char-seq+}
        \lim_{m\rightarrow\infty}\limsup_{k\rightarrow\infty}\dfrac{1}{n_k}\left|\{ 0\leq j <n_k  : \left(\eta_{\mathscr{B}(m)}\right)_j \neq \left(\eta_{\mathscr{B}}\right)_j\}\right|=0.
    \end{equation}
\end{thm}

Let us denote the unique ergodic measure on $\bkshift$ by $\nu_k$. In particular, the characteristic sequence $\eta_{\mathscr{B}(k)}$ is generic for the measure $\nu_k$. Using Davenport--Erd\H{o}s result and Corollary \ref{cor:Besicovitch-summary-symbolic} we deduce the following result.
\begin{cor}
    There exists an ergodic measure $\nu \in \Mse(\bshift)$, which is called the Mirsky measure, such that 
    \begin{equation*}
        \lim_{m\rightarrow\infty}\dbarm\left(\nu_m,\nu\right)=0,
    \end{equation*}
    and $\eta_{\mathscr{B}}$ is a quasi-generic for the Mirsky measure $\nu$ along the subsequence $(n_k)_{k=1}^\infty$. Furthermore, the Mirsky measure has rational discrete spectrum, in particular its entropy $h(\nu)$ vanishes.
    \end{cor}
\begin{proof}
The equation \eqref{eq:conv-of-char-seq+} in Theorem \ref{thm:Davenport-Erdos+} means that there is a strictly increasing sequence $(n_k)_{k=1}^\infty$ such that
\[
\dbar_{(n_k)}(\eta_{\mathscr{B}(m)},\eta_{\mathscr{B}})\to 0\quad\text{ as }m\to\infty.
\]
By Lemma \ref{lem:dbar-D_B-equiv} it follows that condition  
\eqref{cond:char-rho-bar-conv-iii} from Corollary~\ref{cor:char-rho-bar-conv} holds.
We conclude the proof by applying Corollary \ref{cor:Besicovitch-summary-symbolic}.
\end{proof}
Similarly, one can apply $\dbar$ and $\dbarm$ to study the invariant measures of the family of shift spaces given in \cite{KKL}.

The next two applications use results from the current paper. Their authors
had access to draught versions of the current manuscript.

\subsection{Density of ergodic discrete rational spectrum measures for topological dynamical systems with the specification property} Alexandre Trilles, in his PhD thesis \cite{Trilles2025}, proves the following result: \emph{If $(X,T)$ is a topological dynamical system with  the specification property, then the set of ergodic measures with full support and discrete rational spectrum is dense in $\MT(X)$.}
This strengthens a result of Sigmund \cite{Sigmund70}, who proved the density of zero entropy fully supported ergodic measures.

The specification property enables the construction of periodic points whose orbits can approximately trace given segments of orbits, we skip the definition here. The proof uses the specification property to construct a sequence of periodic points 
that converges in the Besicovitch pseudometric to a generic point of ergodic measure, being close to a fixed, predetermined invariant measure in the weak$^*$ topology. Then, the toolbox developed here shows that this ergodic invariant measure, close to the fixed one, has a rational discrete spectrum.

Using the same idea, one can prove that
for every $\beta>1$, the set of ergodic measures with full support and discrete rational spectrum is dense in $\MT(X_\beta)$, where $X_\beta$ is the $\beta$-shift.

\subsection{Density of ergodic measures among all invariant measures for topological dynamical systems with the vague specification property}
Recently, \cite{BJK} proved that any topological dynamical system with the vague specification property has ergodic measures that are dense in the space of all its invariant measures. The paper extends methods from \cite{KKK2} and \cite{CKKK} from symbolic to general topological dynamical systems relying on the $\bar{\rho}$ metric and its properties presented here.

First, the authors of \cite{BJK} construct a sequence of approximating systems for the natural extension $(X_T,S)$ of a topological dynamical system $(X,T)$. Each approximating topological dynamical system consists of $\delta$-chains for some $\delta>0$. It is shown that the periodic specification property holds for these approximations, provided the initial system is chain-mixing. The specification property ensures the density of ergodic measures in the simplex of invariant measures of each approximation.
The crucial technical step is to demonstrate that the simplices of invariant measures of these approximating systems converge to the simplex of invariant measures of the original system in the Hausdorff metric induced by $\bar{\rho}$. It remains to note that this strong form of convergence (in the Hausdorff metric induced by the $\bar{\rho}$-metric rather than just the weak$^*$ topology) ensures that structural properties, such as the density of ergodic measures, transfer from the simplices of invariant measures of approximating systems to the simplex of invariant measures of the limit system.

The key insight is that the vague specification property provides a way to connect orbits of the approximating systems to orbits in the original system while keeping the Besicovitch distance small (here the main result of \cite{CT} is used). This connection then lifts to the level of measures via the relationship between the Besicovitch pseudometric and the $\bar{\rho}$-metric, and then to the Hausdorff distance between simplices of invariant measures corresponding to $\bar{\rho}$. 

This approach allows the authors of \cite{BJK} to prove the density of ergodic measures for a much broader class of systems than was previously known, including minimal and proximal systems where traditional approaches fail. In particular, the examples discussed in \cite{BJK} show that the results apply to non-symbolic dynamical systems; hence, the use of $\bar{\rho}$ is inevitable.

We believe that our result will prove to be useful, and there are more applications to come. For example, Lorenzo J.~D\'{\i}az, Katrin Gelfert, and Michał Rams, in a series of papers, studied ergodic measures of some skew-products aiming at understanding invariant measures of partially hyperbolic diffeomorphisms.  In that setting, some measures are naturally approximated by measures whose support is contained in some horseshoe. With some additional assumptions, one can prove that the approximation means here ``up to an arbitrarily small $\bar{\rho}$ distance''. 

\section*{Acknowledgements}
All authors contributed to the paper equally. Research by SB was supported by the National Science Centre (NCN) Sonata Bis grant no. 2019/34/E/ST1/00237;
research by DK was supported by the NCN Opus grant no. 2022/47/B/ST1/02866;
MEC was supported by the NCN Polonez Bis grant no. 2021/43/P/ST1/02885 and 
NCN Polonez Bis 3 grant No. 2022/47/P/ST1/00854 (H2020 MSCA GA No. 945339); 
PO was supported by the NCN Opus grant no. 2019/35/B/ST1/02239. The financial support provided by the aforementioned institutions is gratefully acknowledged. We would like to thank Mariusz Lemańczyk for many insightful comments that helped us improve this paper.

\bibliographystyle{plain}
\bibliography{bibliography.bib}

\end{document}